\documentclass[a4paper,10pt,oneside,reqno]{amsart}
\usepackage{amsmath}
\usepackage[foot]{amsaddr}
\usepackage[utf8]{inputenc}
  \usepackage{paralist}
  \usepackage{graphics} 
  \usepackage{epsfig} 
\usepackage{graphicx}  \usepackage{epstopdf}
 \usepackage[colorlinks=true]{hyperref}
\hypersetup{urlcolor=blue, citecolor=red}

\newtheorem{theorem}{Theorem}[section]

\newtheorem{lemma}[theorem]{Lemma}
\newtheorem{proposition}[theorem]{Proposition}

\theoremstyle{definition}
\newtheorem{definition}[theorem]{Definition}
\newtheorem{remark}[theorem]{Remark}
\newtheorem{ass}[theorem]{Assumption}
\newtheorem{example}[theorem]{Example}

\newcommand{\setdef}[2]{\left\{\ #1\ \left|\ \vphantom{#1} #2\ \right.\right\}}

\newcommand{\cA}{\mathcal{A}}

\newcommand{\cD}{\mathcal{D}}
\newcommand{\cF}{\mathcal{F}}
\newcommand{\cH}{{\mathcal{H}}}

\newcommand{\cL}{\mathcal{L}}
\newcommand{\cR}{\mathcal{R}}

\newcommand{\cT}{\mathcal{T}}

\newcommand{\N}{{\mathbb{N}}}
\newcommand{\R}{{\mathbb{R}}}
\newcommand{\C}{{\mathbb{C}}}
\newcommand{\K}{{\mathbb{K}}}
\newcommand{\I}{{\mathbb{I}}}

\newcommand{\Af}{\mathfrak{A}}
\newcommand{\Bf}{\mathfrak{B}}
\newcommand{\Cf}{\mathfrak{C}}

\newcommand{\st}{\ |\ }

\newcommand{\yref}{y_{\rm ref}}

\newcommand{\syst}{\mathfrak{S}=(\Af,\Bf,\Cf)}

\newcommand{\ddt}{\tfrac{\text{\normalfont d}}{\text{\normalfont d}t}}
\newcommand{\pddt}{\tfrac{\partial}{\partial t}}

\newcommand{\ds}[1]{{\rm \, d} #1 \,}

\newcommand{\scpr}[2]{\left\langle #1,#2\right\rangle}

\newcommand{\fa}{\mathfrak{a}}
\DeclareMathOperator{\re}{Re}
\DeclareMathOperator{\divg}{div}
\DeclareMathOperator{\grad}{grad}
\DeclareMathOperator*{\esssup}{ess\,sup}

\def\tb#1{\textcolor[rgb]{0.00,0.00,0.00}{#1}}

\title[Funnel control for boundary control systems] 
      {Funnel Control for  boundary control systems}

\author[Puche, Reis, Schwenninger]{}

\subjclass{Primary: 93C20, 93C40; Secondary: 47H06.}
 \keywords{Nonlinear Feedback, Funnel Control, Boundary Control Systems, Hyperbolic PDEs, Parabolic PDEs, Dissipative Operators.}

 \email{marc.puche@uni-hamburg.de}
 \email{timo.reis@uni-hamburg.de}
 \email{felix.schwenninger@uni-hamburg.de}

\thanks{$^*$ Corresponding author: Marc Puche}

\begin{document}

\maketitle

\centerline{\scshape Marc Puche$^{*,\dagger}$, Timo Reis$^\dagger$ and Felix L.~Schwenninger$^{\dagger,\ddagger}$}
\bigskip
{\centering
\begin{minipage}{0.5\linewidth}
{\footnotesize
\centerline{$^\dagger$University of Hamburg}
 \centerline{Bundesstra\ss e 55}
   \centerline{20146 Hamburg}
   \centerline{Germany}
} 
\end{minipage}
\begin{minipage}{0.5\linewidth}
\centering
{\footnotesize
	\centerline{$^\ddagger$University of Twente}
	\centerline{P.O.~Box 217}
	\centerline{7500AE Enschede}
	\centerline{The Netherlands}
}

\end{minipage}}

\medskip

\bigskip


\begin{abstract}
We study a nonlinear, non-autonomous feedback controller applied to boundary control systems. Our aim is to track a given reference signal with prescribed performance. Existence and uniqueness of solutions to the resulting closed-loop system is proved by using nonlinear operator theory. We apply our results to both hyperbolic and parabolic equations.
\end{abstract}
\section{Introduction}

In this paper we consider a class of \emph{boundary control systems} (\emph{BCS}) of the form
\begin{align*}
\dot{x}(t)&=\Af x(t),\quad t>0,\quad x(0)=x_0,\\
u(t)&=\Bf x(t),\\
y(t)&=\Cf x(t),
\end{align*}
where  $\Af,\Bf,\Cf$ are linear operators. {For Hilbert spaces $U$ and $X$, the $U$-valued functions $u$ and $y$ are interpreted as the input and the measured output $y$, respectively, whereas} $x$ is called the state of the system. Typically, $\Af$ is a differential operator on the state space $X$ and $\Bf,\Cf$ are evaluation operators of the state at the boundary of the spatial domain, that is, the domain of the functions lying in $X$.

The aim of this paper is to develop an adaptive controller for boundary control systems which, roughly speaking, achieves the following goal:
\begin{quote}
\it For any prescribed reference signal $\yref\in W^{2,\infty}([0,\infty){,U})$, the output $y$ of the system tracks $\yref$ in the sense that
 the transient behavior of the error $e(t):=y(t)-\yref(t)$ is controlled.
\end{quote}
Shortly, we will elaborate on the class of possible reference signals and the meaning of ``controlling the transient behavior'' in more detail.
The goal will be achieved by using a \emph{funnel controller}, which, in the simplest case, has the form
\[u(t)=-\frac{1}{1-\varphi(t)^2\|e(t)\|^2}\, e(t)\]
for some positive function $\varphi$.
Under this feedback, the error is supposed to evolve in the  \emph{performance funnel}
$$\cF_\varphi:=\{(t,e)\in[0,\infty)\times {U}\st\varphi(t)\|e\|<1 \}$$
and would hence satisfy
$$\|e(t)\|\leq\varphi(t)^{-1},\quad\text{ for all }t\geq0.$$
In fact, if $\varphi$ tends asymptotically to a large value $\lambda$, then the error remains bounded by $\lambda^{-1}$, see Fig.~\ref{fig:funnel} {which illustrates the case of scalar-valued input and output}.
\begin{figure}[ht!!]
	\centering
	\includegraphics[width=200pt]{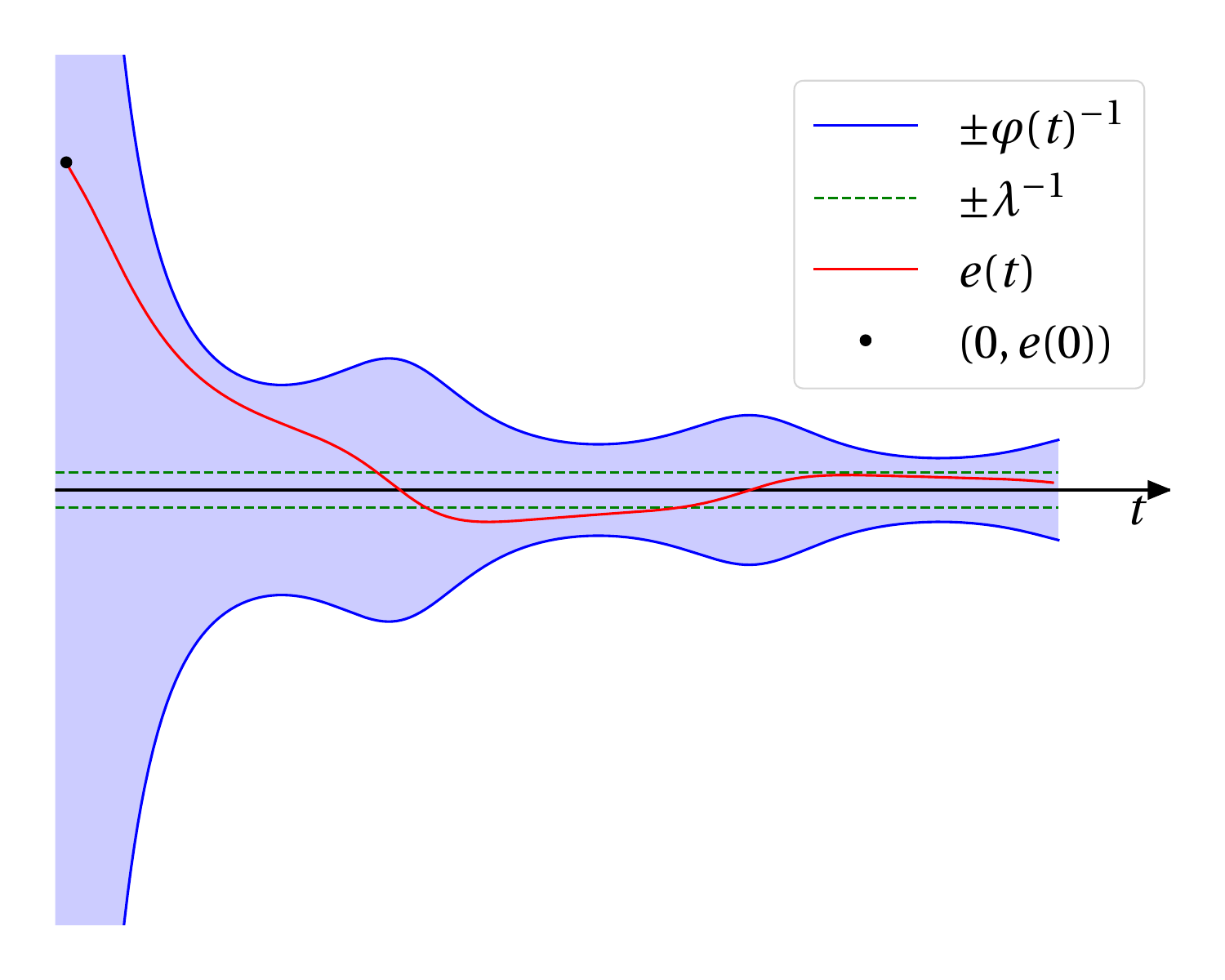}
	\caption{Error evolution in a funnel $\cF_{\varphi}$ with boundary $\varphi(t)^{-1}$.}\label{fig:funnel}
\end{figure}

{Funnel control} was first introduced in \cite{IlchmannRyan}, where feasibility of the funnel controller for a class of functional differential equations has been shown. These encompass infinite dimensional systems with very restrictive assumptions on the operators involved, {a special class of nonlinear delay systems}. In fact, there finite-dimensional linear ``prototype'' systems with relative degree one are treated. The relative degree is a well-known magnitude for finite-dimensional systems and can roughly be understood as the number of times one needs to differentiate the output $y$ {in order to derive an equation relating $y$ and the input $u$.}  
 This quantity turned out to be relevant when considering the funnel controller and has been used to generalize the results of \cite{IlchmannRyan}. For instance, in \cite{Ilch07}, the funnel controller was proved to be applicable for systems with known but arbitrary relative degree. The problem is that the ansatz used there  requires very large powers of the gain factor. This problem has been overcome in \cite{BergHoang} by introducing a~funnel controller which involves derivatives of the output and reference signal, and feasibility of this controller in the case of nonlinear finite-dimensional systems with \emph{strict relative degree} with \emph{stable internal dynamics} has been proven. The funnel control for infinite-dimensional systems has so far only attracted attention in special configurations \cite{BergPuchSchw,IlchmannSelig,ReisSelig}. The recent article \cite{BergPuchSchw} deals with a linearized model of a moving water tank by showing that this system belongs to the class being treated in \cite{BergHoang}, see also \cite{BergPuchSchw20}. In \cite{IlchmannSelig}, a class of infinite-dimensional systems has been considered where feasibility of the funnel controller can be proved in a similar way as for finite-dimensional systems. More precisely, this class consists of systems which possess a so-called {\em Byrnes-Isidori form} via bounded and boundedly invertible state space transformation. The existence of such a~form however requires that the control and observation operators fulfill very strong boundedness conditions, which in particular exclude boundary control and observation. Funnel control of a~heat equation with Neumann boundary control and co-located Dirichlet output has been treated in \cite{ReisSelig}. The proof of feasibility of funnel control uses the spectral properties of the Laplacian, whence this technique is hardly transferable to further classes of boundary control systems.

We consider a~class of {\em boundary control systems} which satisfy a certain energy balance \cite{BastCoro16,JacobZwart}.
The feedback law of the funnel controller naturally induces a nonlinear closed-loop system. For the corresponding solution theory, the concept of (nonlinear) m-dissipative operators in a Hilbert space will play an important role. For an appropriate introduction to this classical topic we refer to~\cite{Kato,Miyadera,Showalter}. 

The paper is organized as follows. In Section \ref{sec:systemclass} we introduce the system class that is the subject of our results. In Section \ref{sec:main} we present the details about the controller and present the main results which refer to the applicability of the funnel controller to the considered system class. In Section~\ref{sec:systex} we present some examples of partial differential equations for which the funnel controller is applicable. Section \ref{sec:proof} contains the proof of the main results together with some preliminary auxiliary results. We provide numerical simulations in Section \ref{sec:sim}.



The expression $\overline{S}$ indicates the closure of a~set $S$. 
The norm in a~normed space {$X$} will be denoted by $\|\cdot\|_{{X}}$ or {simply} $\|\cdot\|$, if clear from context. 
Analogously, the scalar product of an inner product space {$X$} will be denoted by $\scpr{\cdot}{\cdot}_X$ or $\scpr{\cdot}{\cdot}$. \tb{Scalar products are linear in the first argument. Likewise, the duality products $\scpr{\cdot}{\cdot}_{X',X}$ considered in this article are linear in the first argument, and conjugate linear in the second argument, and we set $\langle x,x'\rangle_{X,X'}=\overline{\langle x',x\rangle_{X',X}}$.}\\
{The open unit ball in a~normed space will be denoted by $B_1(0)$.} The space $\K^n$ {(where $\K$ stands for either the field of real or complex numbers)} is {always} provided with the Euclidean inner product. {Normed and inner product spaces can be always real or complex throughout this article. However, we have the following convention: if several spaces are involved within a particular result in this work, then the underlying fields will be assumed to be the same for all these spaces.}

The domain of a (possibly nonlinear) operator $A$ is denoted by $\cD(A)$, $\ker A$ refers to the kernel and $\cR(A)$ stands for the range of $A$. \tb{If $D\subset \cD(A)$, then $A|_{D}$ refers to the restriction of $A$ to $D$}. Given two Banach spaces $X,Y$, the set of linear bounded operators from $X$ to $Y$ will be denoted by $\cL(X,Y)$ and in the case $X=Y$ simply by $\cL(X)$. The identity operator on the space $X$ is $I_X$, or just $I$, if clear from context. We further write $I_m$ instead of $I_{\K^m}$. The symbol $A^*$ stands for the adjoint of a~linear operator $A$. In particular,
$M^\ast\in\K^{n\times m}$ is the transposed of the complex conjugate of $M\in\K^{m\times n}$.


Lebesgue and Sobolev spaces from a measurable set $\Omega\subset\R^d$ will be denoted by $L^{p}(\Omega)$ and $W^{k,p}(\Omega)$.
For a~domain $\Omega\subset\R^d$ with sufficiently smooth boundary $\partial\Omega:=\Gamma$, we denote by $W^{k,p}(\Gamma)$ the Sobolev space at the boundary \cite{Adam75}. The set of infinitely often differentiable functions from $\Omega$ with compact support will be denoted by $C_0^\infty(\Omega)$. 
We identify spaces of $\K^n$-valued functions with the Cartesian product of spaces of scalar-valued functions, such as, for instance $(W^{k,p}(\Omega))^n\cong W^{k,p}(\Omega;\K^n)$.
For an~interval $J\subset\R$ and a Banach space $B$, we set
\[W^{k,\infty}(J;B):=\{f\in L^\infty(J;B)\st f^{(j)}\in L^\infty(J;B),j=0,\dots,k \},\]
which is to be understood in the Bochner sense \cite{Diestel77}. The space $W^{k,\infty}_{\rm loc}(J;B)$ consists of all $f$ whose restriction to any compact interval $K\subset J$ is in $W^{k,\infty}(K;B)$.
\pagebreak[3]
\section{System class}
\label{sec:systemclass}
In the following we introduce our system class, define our controller and discuss the  solution concept to the resulting nonlinear feedback system.
\begin{definition}[System class]\label{def:system_class}
	Let $X$ and $U$ be Hilbert spaces and let $m\in\N$ be given. \tb{Let $\Af:\cD(\Af)\rightarrow X$ be a bounded linear operator, where $\cD(\Af)$ is a~Hilbert space with norm $\|\cdot\|_{\cD(\Af)}$ which is continuously embedded in $X$.
	} Furthermore, let $\Bf,\Cf\in\cL(\cD(\Af),U)$. To the triple $\syst$ we associate the system
	\begin{equation}\label{eq:syst_class}
	\left.\begin{aligned}
	\dot{x}(t)&=\Af x(t),\quad x(0)=x_0,\\\
	u(t)&=\Bf x(t),\\
	y(t)&=\Cf x(t),
	\end{aligned}\right\}
	\end{equation}
{and we call} it a \emph{boundary control system (BCS)}.
\end{definition}
In the sequel we specify the system class.
\begin{ass}\label{ass:system_class}
Let a~ BCS $\syst$ be given.
	\begin{enumerate}[(i)]
		\item\label{ass:1} The system is \emph{(generalized) impedance passive}, i.e., there exists $\alpha\in\R$ such that
		\begin{equation}\label{eq:diss}
		\re\langle \Af x, x\rangle_X\leq\re\langle \Bf x,\Cf x\rangle_{{U}}+\alpha\|x\|_X^2\;\text{ for all $x\in\cD(\Af)$}.
		\end{equation}\label{ass:passivity}
		\item\label{ass:3} {There exists some $\beta\ge\alpha$, such that 
		$\cR(\Af|_{\ker\Cf}-\beta I)=X$.}
		\item\label{ass:CBonto} The operator%
		\begin{equation}\begin{bmatrix}
		\Bf\\
		\Cf
		\end{bmatrix}:\cD(\Af)\to {U\times U}\label{eq:CBop}
\end{equation}
		is surjective.
	\end{enumerate}
\end{ass}
{Next we discuss our assumptions in detail. In particular we show that $\Af|_{\ker\Cf}$ generates a~strongly continuous semigroup and, moreover, the above assumptions imply that, in the case where $U$ is finite-dimensional, Assumption \ref{ass:system_class}\eqref{ass:3} can be replaced with the condition that $\Af|_{\ker\Bf}$ is the generator of a~strongly continuous semigroup.
\begin{proposition}\label{prop:ass}
Let a~ BCS $\syst$ be given with the properties specified in Assumptions~\ref{ass:system_class}\eqref{ass:1}\&\eqref{ass:CBonto}. Then the following two statements are equivalent:
\begin{itemize}
\item[(i)] Assumption~\ref{ass:system_class}\eqref{ass:3} holds.
\item[(ii)] $\Af|_{\ker\Cf}$ generates a strongly continuous semigroup on $X$.
\end{itemize}
Moreover, the following two statements are equivalent.
\begin{itemize}
\item[(iii)] There exists some $\beta\ge\alpha$ such that   $\cR(\Af|_{\ker\Bf}-\beta I)=X$.
\item[(iv)] $\Af|_{\ker\Bf}$ generates a strongly continuous semigroup on $X$.
\end{itemize}
If, additionally, $\dim U<\infty$, then all of the above statements (i)--(iv) are equivalent.
\end{proposition}}
{Whereas the proofs of the equivalence between (i) and (ii) as well as that between (iii) and (iv) are brief (see also the forthcoming Remark~\ref{rem:system_class}), our proof of the equivalence between all of these four statements under the additional assumption that $U$ is finite-dimensional requires a~technical lemma and is therefore postponed to Section~\ref{sec:proof}.}
\pagebreak[3]
\begin{remark}\label{rem:system_class}
\hspace{1em}
	\begin{enumerate}[(a)]
\item\label{rem:system_class1} {The equivalence between (i) and (ii) as well as the one between (iii) and (iv) are direct consequences of the Lumer–Phillips theorem, see e.g.\ \cite[Cor.~II.3.20]{EngeNage00}. Note that the density of the domain follows implicitly from the dissipativity and the the range condition as the considered spaces are reflexive. Also note that the Lumer–Phillips theorem further implies that the semigroups $T_i(\cdot),T_o(\cdot):[0,\infty)\to \cL(X)$ which are respectively generated by $\Af|_{\ker\Bf}$ and $\Af|_{\ker\Cf}$ satisfy $\|T_i(t)\|\leq e^{\alpha t}$ and $\|T_o(t)\|\leq e^{\alpha t}$ for all $t>0$. In particular, the semigroups are contractive, if $\alpha\leq0$ and moreover exponentially stable, if $\alpha<0$. The semigroup $T_i(\cdot)$ describes the free dynamics, i.e., the evolution of the state when $u\equiv0$. Likewise, the semigroup $T_o(\cdot)$ specifies the remaining dynamics when the output is vanishing constantly, i.e., $y\equiv0$. In the theory of adaptive control, the latter is known as {\em zero dynamics} \cite{IlchmannRyan}. In particular, the zero dynamics are exponentially stable, if $\alpha<0$.}
{\item\label{rem:system_class2} The Lumer–Phillips theorem \cite[Thm.~3.15]{EngeNage00} further implies that, if Assumptions~\ref{ass:system_class}\eqref{ass:passivity}\&\eqref{ass:3} hold, then $\cR(\Af|_{\ker\Cf}-\beta I)=X$ for all $\beta\ge\alpha$. Likewise, under Assumption~\ref{ass:system_class}\eqref{ass:passivity}, the existence of some $\beta\ge\alpha$ with $\cR(\Af|_{\ker\Bf}-\beta I)=X$ is equivalent to $\cR(\Af|_{\ker\Bf}-\beta I)=X$ for all $\beta\ge\alpha$.}
{\item\label{rem:system_class2a} For a~BCS $\syst$ satisfying Assumption~\ref{ass:system_class} it holds that $\Cf$ is onto by \eqref{ass:CBonto}, $\ker \Cf$ is dense in $X$ by \eqref{ass:1}\&\eqref{ass:3}, and, by \eqref{ass:3}, $\cR(\Af|_{\ker\Cf}-\beta I)=X$ for  some $\beta\geq\alpha$. Further, \eqref{ass:passivity} implies that $\ker (\Af|_{\ker\Cf}-\beta I)=\{0\}$. Then $\dot{x}(t)=\Af x(t)$, $y(t)=\Cf x(t)$ is a~boundary control system in the sense of \cite[Def.~10.1.1]{TucsnakWeiss}.\\
Likewise, if the BCS $\syst$ satisfies Assumption~\ref{ass:system_class}\eqref{ass:passivity}\&\eqref{ass:CBonto} as well as $\cR(\Af|_{\ker\Bf}-\beta I)=X$ for some $\beta>\alpha$, then $\dot{x}(t)=\Af x(t)$, $u(t)=\Bf x(t)$ is a~boundary control system in the sense of \cite[Def.~10.1.1]{TucsnakWeiss}.}
		\item\label{rem:system_class3} The operator \eqref{eq:CBop}  is onto if, and only if,
there exist $P,Q\in\cL({U},\cD(\Af))$ with
		\begin{equation}\label{eq:onto}
		\begin{bmatrix}
		\Bf\\
		\Cf
		\end{bmatrix}\begin{bmatrix}
		P & Q
		\end{bmatrix}=\begin{bmatrix}
		I_{{U}} & 0\\
		0 & I_{{U}}
		\end{bmatrix}.
		\end{equation}
\item\label{item:wpls} An oftentimes considered class in infinite-dimensional linear systems theory is that of {\em well-posed linear systems}, see e.g.~\cite{Staff05}. That is, {in addition to $\Af|_{\ker\Bf}$ generating a~strongly continuous semigroup}, the controllability map, the observability map and the input-output map are bounded operators. Note that we do not impose such a~well-posedness assumption throughout this article.
		\item\label{item:closedAf} \tb{Note that for a general boundary control system $\syst$, the graph of the operator $\Af$ need not be closed in $X\times X$. However, if conversely, $\Af$ has a closed graph in $X\times X$, then  $\cD(\Af)$ equipped with the graph norm $\|x\|_{\cD(\Af)}=\left(\|x\|_{X}+\|\Af x\|_{X}\right)^{1/2}$ is a Hilbert space, which is continuously embedded in $X$ and $\Af:\cD(\Af)\to X$ is bounded.}
			\end{enumerate}
\end{remark}

\begin{example}
	There are several systems which fit in our description. A class of examples of hyperbolic type is given by so-called {\it port-Hamiltonian systems} such as the {\it lossy transmission line}
	\begin{align*}
	\tfrac{\partial V}{\partial \zeta}(\zeta,t)&=-L\tfrac{\partial I}{\partial t}(\zeta,t)-RI(\zeta,t),\\
	\tfrac{\partial I}{\partial \zeta}(\zeta,t)&=-C\tfrac{\partial V}{\partial t}(\zeta,t)-GV(\zeta,t),\\
	u(t)&=\begin{pmatrix}
	V(a,t)\\
	V(b,t)
	\end{pmatrix},\qquad
	y(t)=\begin{pmatrix}
	I(a,t)\\
	-I(b,t)
	\end{pmatrix},
	\end{align*}
	where $V$ and $I$ are the voltage and the electric current at a point $\zeta$ of a segment $(a,b)$ over the time $t$. A precise definition of port-Hamiltonian systems will be given in Section~\ref{subsec:phs}.\smallskip
	
		In  Section \ref{subsec:ps} we will also apply the theoretical results to parabolic systems given through a general second-order elliptic operator on a \tb{bounded domain $\Omega$ with Lipschitz boundary $\Gamma$}. A particular case is the heat equation,
	\begin{align*}
	\tfrac{\partial x}{\partial t}(t,\zeta)&=\Delta x(t,\zeta),\\
	 u(t)\,\tb{w(\zeta)}&=\nu^\top\grad x(t,\zeta)|_{\Gamma},\qquad	y(t)=\int_{\Gamma}{x(t,\zeta)\,\tb{\overline{w(\xi)}}\,\mathrm{d}\zeta},
	\end{align*}
	where $\nu$ refers to the outward normal along the boundary, \tb{ and $w\in W^{-1/2,2}(\Gamma)$ is a~given weight}. The control variable is thus the heat flux at the boundary and the observation is the total temperature along the boundary \tb{weighted with $w$. Note that this causes that the Neumann boundary value is a~scalar multiple of the weight $w$ for all times $t$}.
\end{example}

\section{Funnel controller}
\label{sec:main}

The following definition presents the cornerstone of our controller, the class of admissible funnel boundaries.

\begin{definition}
	Let
	\begin{align*}
	\Phi&:=\left\{\varphi\in W^{2,\infty}([0,\infty))\st \varphi\text{ is real-valued with }\inf_{t\ge0}\varphi(t)>0\right\}.
	\end{align*}
	 With $\varphi\in\Phi$  
	we associate the \emph{performance funnel}
	$$\mathcal{F}_\varphi:=\{(t,e)\in[0,\infty)\times{U}\st\varphi(t)\|e\|<1\}.$$
	In this context we refer to $1/\varphi(\cdot)$ as \emph{funnel boundary}, see also Fig.~\ref{fig:funnel}.
\end{definition}
Now we define our controller, which is a slight modification of the original controller introduced in \cite{IlchmannRyan}.
 For $x_0\in\cD(\Af)$, we define the funnel controller as
\begin{equation}\label{eq:funnel_controller}
u(t)=\left(u_0+\frac{1}{1-\varphi_0^2\|e_0\|^2}\, e_0\right)p(t)-\frac{1}{1-\varphi(t)^2\|e(t)\|^2}\, e(t),
\end{equation}
where $\varphi_0=\varphi(0)$,  {$e(t)=y(t)-\yref(t)$}, $e_0:=\Cf x_0-\yref(0)$, $u_0:=\Bf x_0$, and $p$ is a function with compact support and $p(0)=1$. In the following we collect assumptions on the functions involved in the funnel controller and the initial value of the BCS $\syst$. Particularly, this includes that the expressions $\Cf x_0$ and $\Bf x_0$ are well-defined.

\begin{ass}[Reference signal, performance funnel, initial value]\label{ass:contr}
The initial value $x_0$ of the~BCS $\syst$ and the functions in the~controller \eqref{eq:funnel_controller} fulfill
	\begin{enumerate}[(i)]
		\item $\yref\in W^{2,\infty}([0,\infty);{U})$;
		\item $p\in W^{2,\infty}([0,\infty))$ with compact support and $p(0)=1$;
		\item $x_0\in\cD(\Af)$ and $\varphi\in\Phi$ with $\varphi(0)\|\Cf x_{0}-\yref(0)\|<1$.
	\end{enumerate}
\end{ass}
\tb{Observe that by using the funnel controller, if we measure the output with some disturbances $d\in W^{2,\infty}([0,\infty);U)$, that is, $y+d$, this will lead to the situation that we track the signal $\hat{y}_{\rm ref}:=\yref-d$.}
\begin{remark}\label{rem:funnel_class}
Apart from smoothness of the reference signal and performance funnel, Assumption \ref{ass:contr} basically includes two points:
\begin{enumerate}[(a)]
\item The initial value is ``smooth'', i.e., $x_0\in\cD(\Af)$. The reason is that --- especially for hyperbolic systems --- the initialization with $x_0\in X\setminus\cD(\Af)$ might result in a discontinuous output. This effect typically occurs when the semigroup generated by $\Af|_{\ker\Bf}$ is not analytic, such as, for instance, when a~wave equation is considered.
\item The output of the system at $t=0$ is already in the performance funnel.
\end{enumerate}
The funnel controller \eqref{eq:funnel_controller} differs from the classical one in \cite{IlchmannRyan} by the addition of the term
\[\left(u_0+\tfrac{1}{1-\varphi_0^2\|e_0\|^2}\, e_0\right)p(t)\]
for some (arbitrary) smooth function with $p(0)=1$ and compact support. This ensures that the controller is consistent with the initial value, that is, $u$ in \eqref{eq:funnel_controller} satisfies
\[u(0)=\left(u_0+\tfrac{1}{1-\varphi_0^2\|e_0\|^2}\, e_0\right)p(0)-\tfrac{1}{1-\varphi(0)^2\|e(0)\|^2}\, e(0)=u_0= \Bf x_0=\Bf x(0).\]
The funnel controller therefore requires the knowledge of the ``initial value of the input'' $u_0=\Bf x_0$. This means that, loosely speaking, the ``actuator position'' has to be known \tb{exactly} at the initial time, which is ---in the opinion of the authors--- no restriction from a practical point of view.\\ 
We would like to emphasize that the application of the funnel controller does not need any further ``internal information'' on the system, such as system parameters or the full knowledge of the initial state.
\end{remark}

The funnel controller (\ref{eq:funnel_controller}) applied to a~BCS $\syst$ results in the closed-loop system
\begin{subequations}\label{eq:funnel_problem}
	\begin{equation}\label{eq:funnel_problem1}
	\left.\begin{aligned}
	\dot{x}(t)&=\Af x(t),\qquad\qquad x(0)=x_0,\\
	\Bf x(t)&=u(t),\\
	\Cf x(t)&=y(t),\\
	e(t)&=y(t)-\yref(t), \quad e_0=\Cf x_{0}-\yref(0),\quad\varphi_0=\varphi(0),\\
	u(t)&=(\Bf x_0+\psi(\varphi_0,e_0))p(t)-\psi(\varphi(t),e(t)),
	\end{aligned}\right\}
	\end{equation}
	where
	\begin{equation}\label{eq:funnel_problem2}
	\begin{aligned}
	\psi(\varphi,e)&:=\frac{1}{1-\varphi^2\|e\|^2}\, e,\qquad
	\cD(\psi):=\{(\varphi,e)\in(0,\infty)\times{U}\st\varphi\|e\|<1\}.
	\end{aligned}
	\end{equation}
\end{subequations}
We see immediately that the closed-loop system is nonlinear and time-variant. In the sequel we present our main results which state that the funnel controller ``does its job'' in a~certain desired way. Note that this result includes the specification of the solution concept with which we are working. First, we show that the funnel controller applied to any system fulfilling Assumption~\ref{ass:system_class} has a~solution. Such a~solution however might not be bounded on the infinite time horizon. Thereafter, we show that boundedness on $[0,\infty)$ is guaranteed, if the constant $\alpha$ in the energy balance \eqref{eq:diss} is negative. The proofs of these results can be found in Section~\ref{sec:proof}.
\pagebreak[2]
\begin{theorem}[Feasibility of funnel controller, arbitrary $\alpha$]\label{thm:funnel1}
Let $T>0$ and let $\syst$ be a~BCS which satisfies Assumption~\ref{ass:system_class}. Further assume that the initial value $x_0$ and the functions $\yref$, $p$, $\varphi$ fulfill Assumption~\ref{ass:contr}.\\
Then the closed-loop system \eqref{eq:funnel_problem} possesses a~unique solution $x\in W^{1,\infty}([0,T];X)$ in the sense that for almost all $t\in [0,T]$, it holds that
		$\dot{x}$ is continuous at $t$, $x(t)\in\mathcal{D}(\Af)$ and \eqref{eq:funnel_problem} is satisfied.\\
	{Furthermore, the set $\{x_{0}\in D(\Af)\:\colon\: \varphi(0)\|\Cf x_{0}-\yref(0)\|<1\}$ is dense in $X$. }
\end{theorem}

\pagebreak[2]
\begin{remark}\label{rem:sol}
\hspace{1em}
\begin{enumerate}[(a)]
\item The solution concept which is subject of Theorem~\ref{thm:funnel1} is strong in the sense that the weak derivative of $x$ is evolving in the space $X$ and not in some \tb{extrapolation space $X_{-1}\supset X$} as used e.g.\ in \cite{TucsnakWeiss}.
\item \tb{A~direct consequence of the property $x\in W^{1,\infty}([0,T];X)$ of a~solution is} that $\Af x=\dot{x}\in L^{\infty}([0,T];X)$, whence $x\in L^{\infty}([0,T];D(\Af))$. Hence, for $u=\Bf x$ and $y=\Cf x$, it holds that $u,y\in L^{\infty}([0,T];{U})$. By the same argumentation, we see that the continuity of $\dot{x}$ at almost every\ $t\in[0,T]$ implies that $u$ and $y$ are continuous at almost every $t\in[0,T]$.
\item\label{item:sol1} For $T_1<T_2$ consider solutions $x_1$ and $x_2$ of the closed-loop system \eqref{eq:funnel_problem} on $[0,T_1]$ and $[0,T_2]$, respectively. Uniqueness of the solution implies that $x_1=x_2|_{[0,T_1]}$. As~a~consequence, there exists a~unique $x\in  W^{1,\infty}_{\rm loc}([0,\infty);X)$ with the property that, for all $T>0$, $x|_{[0,T]}$ is a~solution of (\ref{eq:funnel_problem}). Accordingly, the input satisfies $u\in  L^{\infty}_{\rm loc}([0,\infty);{U})$. Note that, by the fact that the output evolves in the funnel, we have that $y$ is essentially bounded, that is $y\in  L^{\infty}([0,\infty);{U})$.
\item The properties $u,y, \yref \in L^{\infty}([0,T];{U})$ imply that the error $e=y-\yref$ is uniformly bounded away from the funnel boundary. That is, there exists some $\varepsilon>0$ such that
	$$\varphi(t)\|e(t)\|<1-\varepsilon\;\text{ for almost all $t\in[0,T]$.}$$
\end{enumerate}\end{remark}
Though bounded on each bounded interval, the solution $x$ of the closed-loop system \eqref{eq:funnel_problem} might satisfy
	$$\limsup_{t\to\infty}\|x(t)\|=\infty,\qquad \limsup_{t\to\infty}\|u(t)\|=\infty$$
In the following we show that this unboundedness does not occur when the constant $\alpha$ in \eqref{eq:diss} in Assumption~\ref{ass:system_class} is negative.

\begin{theorem}\label{thm:funnel2}
Let a~BCS $\syst$ be given which satisfies Assumption~\ref{ass:system_class} such that Assumption~\ref{ass:system_class}\eqref{ass:passivity} holds with $\alpha<0$. Assume that the initial value $x_0$ and the functions $\yref$, $p$, $\varphi$ fulfill Assumption~\ref{ass:contr}. Then the solution $x:[0,\infty)\to X$ of the closed-loop system \eqref{eq:funnel_problem} (which exists by Theorem~\ref{thm:funnel1}) fulfills
\[x\in W^{1,\infty}([0,\infty);X)\;\text{ and }\; u=\Bf x\in L^{\infty}([0,\infty);{U}).\]
\end{theorem}

\begin{remark}\label{rem:k0}
When the input and output of a~control system have different physical dimensions, it might be essential that the funnel controller is
dilated by some constant $k_0>0$. More precisely, one might consider the controller
\begin{equation}\label{eq:funnel_controllerk0}
u(t)=\left(u_0+\frac{k_0}{1-\varphi_0^2\|e_0\|^2}\, e_0\right)p(t)-\frac{k_0}{1-\varphi(t)^2\|e(t)\|^2}\, e(t).
\end{equation}
The feasibility of this controller is indeed covered by Theorems~\ref{thm:funnel1} \& \ref{thm:funnel2}, which can be seen by the following argumentation: Consider the BCS $\syst$ with transformed input $\tilde{u}=k_0^{-1} u$. That is, a system $\mathfrak{S}=(\Af,k_0^{-1}\Bf,\Cf)$. Providing $X$ with the equivalent inner product $\langle \cdot ,\cdot \rangle_{\rm new}:=k_0^{-1}\langle \cdot ,\cdot \rangle_{X}$, we obtain
\[		\re \langle \Af x, x\rangle_{\rm new}\leq\re\langle k_0^{-1}\Bf x,\Cf x\rangle_{{U}}+\alpha \|x\|_{\rm new}^2\;\text{ for all $x\in\cD(\Af)$}.\]
Consequently, by Theorem~\ref{thm:funnel1}, the funnel controller
\[\tilde{u}(t)=\Big(\underbrace{\tilde{u}_0}_{=k_0^{-1}u_0}+\frac{1}{1-\varphi_0^2\|e_0\|^2}\, e_0\Big)p(t)-\frac{1}{1-\varphi(t)^2\|e(t)\|^2}\, e(t)\]
results in feasibility of the closed-loop. Now resolving $\tilde{u}=k_0^{-1} u$ in the previous formula, we obtain exactly the controller \eqref{eq:funnel_controllerk0}. Further note that, by the same argumentation together with Theorem~\ref{thm:funnel2}, we obtain that all the trajectories are bounded in the case where $\alpha<0$.
\end{remark}

\section{Some PDE examples}\label{sec:systex}
We now present three different system classes for which we can apply the previously presented results. The first two have state variables which are described by hyperbolic PDEs and the third one by a parabolic PDE.

\subsection{Port-Hamiltonian systems in one spatial variable}
\label{subsec:phs}

The systems considered in this article encompass a~class of port-Hamiltonian hyperbolic systems in one spatial dimension with boundary control and observation, which has been treated in \cite{Augner,AugnerDis,AugnerJacob,BastCoro16,JacobZwart} and is the subject of the subsequent definition. Typically they are considered in a bounded interval $(a,b)\subset\R$. We may consider $\I:=(a,b)=(0,1)$ without loss of generality.
	
\begin{definition}[Port-Hamiltonian hyperbolic BCS in one spatial variable]\label{def:phs}
	Let $N,d\in\N$ and for $k=0,\dots,N$ consider $P_k\in\K^{d\times d}$. We assume that $P_k=(-1)^{k+1}P_k^\ast$ for $k\neq0$ with $P_N$ invertible and $P_0+P_0^*\leq0$. Further, let $W_B,W_C\in\K^{Nd\times2Nd}$ be such that the matrix
	$$W:=\begin{bmatrix}
	W_B\\
	W_C
	\end{bmatrix}\in\K^{2Nd\times2Nd}$$
	is invertible.
	\begin{enumerate}[(a)]
		\item Let $\cH\in L^\infty\left([0,1];\K^{d\times d}\right)$ and assume that there exist constants $m,M>0$ such that  $\cH(\zeta)=\cH(\zeta)^\ast$  and $m I_d\leq \cH(\zeta)\leq M I_d$ for almost every $\zeta\in[0,1]$. We consider $X:=L^2([0,1];\K^d)$ equipped with the scalar product induced by $\cH$,
		\begin{equation}\label{eq:energy_innerprod}
		\langle y,x\rangle_X:=\langle y,\cH x\rangle_{L^2}=\int_0^1{y(\zeta)^\ast\cH(\zeta)x(\zeta)\ds{\zeta}}\quad\forall\, x,y\in L^2([0,1];\K^d).
		\end{equation}
		The \emph{port-Hamiltonian operator} $\Af:\cD(\Af)\subset X\rightarrow X$ is given by
		\begin{subequations}\label{eq:U_operator_dom}
			\begin{equation}\label{eq:U_operator}
			\Af x=\sum_{k=0}^NP_k\tfrac{\partial^k}{\partial\zeta^k}(\cH x)\quad\forall\, x\in\cD(\Af),
			\end{equation}
			with domain
			\begin{equation}\label{eq:U_dom}
			\cD(\Af)=\left\{x\in X\st\cH x\in W^{N,2}([0,1];\K^d)\right\}.
			\end{equation}
		\end{subequations}
		\item Denote the spatial derivative of $f$ by $f'$. For a port-Hamiltonian operator $\Af$ and $x\in\cD(\Af)$ we define the \emph{boundary flow} $f_{\partial,\cH x}\in\K^{Nd}$ and \emph{boundary effort} $e_{\partial,\cH x}\in\K^{Nd}$ by
		\begin{equation}\label{eq:boundary_flow_effort}
		\begin{pmatrix}
		f_{\partial,\cH x}\\
		e_{\partial,\cH x}
		\end{pmatrix}:=R_0\begin{pmatrix}
		(\cH x)(1)\\
		(\cH x)'(1)\\
		\vdots\\
		(\cH x)^{(N-1)}(1)\\
		(\cH x)(0)\\
		(\cH x)'(0)\\
		\vdots\\
		(\cH x)^{(N-1)}(0)
		\end{pmatrix},
		\end{equation}
		where the matrix $R_0\in\K^{2Nd\times 2Nd}$ is defined by
		\begin{equation}\label{eq:R0}
		R_0:=\dfrac{1}{\sqrt{2}}\begin{bmatrix}
		\Lambda & -\Lambda\\
		I_{Nd} & I_{Nd}
		\end{bmatrix},
		\end{equation}
		with
		$$\Lambda:=\begin{bmatrix}
		P_1 & P_2 & \cdots & \cdots & P_N\\
		-P_2 & -P_3 & \cdots & -P_N & 0\\
		\vdots & \vdots &\vdots & \vdots & \vdots\\
		(-1)^{N-1}P_N & 0 &\cdots& 0 & 0
		\end{bmatrix}.
		$$
		\item For a port-Hamiltonian operator $\Af$ we define the \emph{input map} $\Bf:\cD(\Af)\subset X\rightarrow\K^{Nd}$ and the \emph{output map} $\Cf:\cD(\Af)\subset X\rightarrow\K^{Nd}$ as
	\[		\Bf x:=W_B\begin{pmatrix}
		f_{\partial,\cH x}\\
		e_{\partial,\cH x}
		\end{pmatrix},\qquad
		\Cf x:=W_C\begin{pmatrix}
		f_{\partial,\cH x}\\
		e_{\partial,\cH x}
		\end{pmatrix}.\]
		We call $\syst$ a \emph{port-Hamiltonian hyperbolic BCS in one spatial variable} to which we associate \tb{the BCS $\syst$} 
		with state $x(t):=x(t,\cdot)\in X$, $t\geq0$  {and initial value $x_{0}$}.
	\end{enumerate}
\end{definition}
{The above definition directly implies the following property for port-Hamiltonian systems.}

\begin{lemma}\label{lemm:right_inverse}
	{Let $(\Af,\Bf,\Cf)$ be a port-Hamiltonian hyperbolic BCS. Then $\Af$ is closed and $\Bf,\Cf$ are bounded from $D(\Af)$ to $\K^{Nd}$.}
	Further, there exist $P,Q\in \cL(\K^{2Nd},\cD(\Af))$ {such that}
	$$
	\begin{array}{ccccccc}
	\Bf P&=&I_{Nd},&\quad\Bf Q&=&0,\\
	\Cf P&=&0,&\quad\Cf Q&=&I_{Nd}.\\
	\end{array}
	$$
	Consequently, $\Af P,\Af Q\in\cL(\K^{Nd},X)$.
\end{lemma}

\begin{proof}
\tb{The closedness of $\Af$ follows directly from completeness of Sobolev spaces. The boundedness of $\Bf$ and $\Cf$ is an immediate consequence of the Sobolev embedding into the continuous functions.}
	Consider the trace operator $\cT:W^{N,2}([0,1];\K^d)\rightarrow\K^{2Nd}$ as the linear map
	$$
	\cT z=\begin{pmatrix}
	z(1)\\
	z'(1)\\
	\vdots\\
	z^{(N-1)}(1)\\
	z(0)\\
	z'(0)\\
	\vdots\\
	z^{(N-1)}(0)
	\end{pmatrix},
	$$
	so that
	$$\begin{bmatrix}
	\Bf x\\
	\Cf x
	\end{bmatrix}=WR_0\cT \cH x,\qquad \text{where }W=\begin{bmatrix}
	W_B\\
	W_C
	\end{bmatrix}.$$
	Let  $\{e_j\}_{j=1}^{2Nd}$ be the standard orthogonal basis in $\K^{2Nd}$ and choose some functions $f_j\in W^{N,2}([0,1];\K^d)$ with $\cT(f_j)=e_j$ for $j=1,\dots,2Nd$. Since $W,R_0$ are invertible, we can define $M_p,M_q\in\K^{2Nd\times Nd}$ by
	$$M_p=R_0^{-1}W^{-1}\begin{bmatrix}
	I_{Nd}\\
	0
	\end{bmatrix},\quad M_q=R_0^{-1}W^{-1}\begin{bmatrix}
	0\\
	I_{Nd}
	\end{bmatrix}.$$
	Let $M_p,M_q$ be decomposed as
	$$
	M_p=\begin{bmatrix}
	M_{p,1}\\
	\vdots\\
	M_{p,2Nd}
	\end{bmatrix},\quad M_q=\begin{bmatrix}
	M_{q,1}\\
	\vdots\\
	M_{q,2Nd}
	\end{bmatrix},
	$$
	with $M_{p,j},M_{q,j}\in\K^{1\times Nd}$ for $j=1,\dots,2Nd$. Now set for almost every $\zeta\in[0,1]$,
	\begin{align*}
	(Pu)(\zeta)&:=\cH^{-1}(\zeta)\sum_{j=1}^{2Nd}M_{p,j}uf_j(\zeta)\quad \forall u\in\K^{Nd},\\
	(Qy)(\zeta)&:=\cH^{-1}(\zeta)\sum_{j=1}^{2Nd}M_{q,j}yf_j(\zeta)\quad \forall y\in\K^{Nd}.
	\end{align*}
	By construction $P,Q$ have the desired properties.
\end{proof}





The class of impedance passive port-Hamiltonian systems meets the requirements of Assumption \ref{ass:system_class}. We summarize this in the following statement.

\begin{theorem}\label{thm:funnel_phs}
	Any~port-Hamiltonian hyperbolic BCS $\syst$ in one spatial variable satisfies Assumption~\ref{ass:system_class}. If, moreover, there exists some
$\mu>0$ such that $P_0+P_0^*+\mu I$ is negative definite, then Assumption~\ref{ass:system_class}\eqref{ass:passivity} {even} holds for some $\alpha<0$.
\end{theorem}

\begin{proof}

{
Integration by parts gives
	\begin{equation}\label{eq:impedance_passive}
	\re\langle\Af x,x\rangle_X\leq\re\langle\Bf x,\Cf x\rangle_{\K^{Nd}}+\re\langle P_0\cH x,\cH x\rangle_{L^2}\quad \forall\, x\in\cD(\Af).
	\end{equation}
	}Since $P_0+P_0^*\leq0$, it follows that the BCS fulfills Assumption~\ref{ass:system_class}\eqref{ass:passivity} with $\alpha\leq0$.
	Further,  $\Af|_{\ker\Bf}$ generates a (contractive) semigroup by \cite[Thm.~2.3]{AugnerJacob}{, whence $\Af|_{\ker\Cf}$ generates a (contractive) semigroup by Proposition \ref{prop:ass} and the fact that $\dim U<\infty$. Thus }$\syst$ satisfies Assumption~\ref{ass:system_class}\eqref{ass:3}.
Lemma \ref{lemm:right_inverse} {guarantees the boundedness of $\Bf$ and $\Cf$ as well as} the existence of $P,Q$ such that (\ref{eq:onto}) holds. This implies that the condition in Assumption~\ref{ass:system_class}\eqref{ass:CBonto} is satisfied by $\syst$.\\
If, moreover, $P_0+P_0^*+\mu I$ is negative definite for some $\mu>0$, then we can conclude from \eqref{eq:impedance_passive} that
Assumption~\ref{ass:system_class}\eqref{ass:passivity} holds with $\alpha:=-\mu m/(2M)$, where $m,M>0$ are given in Definition~\ref{def:phs}.
\end{proof}
{By Theorem~\ref{thm:funnel_phs}, Theorems~\ref{thm:funnel1} \& \ref{thm:funnel2} can be applied to systems port-Hamiltonian hyperbolic BCSs}. {Indeed}, if the initial value $x_0$ and the functions $\yref$, $p$, $\varphi$ fulfill Assumption~\ref{ass:contr}, the application of the funnel controller \eqref{eq:funnel_controller} results in a~unique global solution $x\in W^{1,\infty}_{\rm loc}([0,\infty);X)$ in the sense of Theorem~\ref{thm:funnel1}. If, moreover, $P_0+P_0^*+\mu I$ is negative definite for some $\mu>0$, then $x,\dot{x}$ and $u$ are moreover essentially bounded by Theorem~\ref{thm:funnel2}.

\subsection{Parabolic systems with Neumann boundary control}
\label{subsec:ps}

A particular case of the boundary controlled heat equation was already discussed in \cite{ReisSelig}, with a slightly different funnel controller. Here we present a parabolic problem which comprises linear parabolic equations with Neumann boundary control and Dirichlet boundary observation. Namely, for a~domain $\Omega\subset\R^n$ with outward normal $\nu$ along the boundary $\Gamma$, and some given functions $a$, $b$, $c$ and $w_i$ for $i=1,\ldots,m$, we consider a~model problem of the form
\begin{equation}\label{eq:str_parab}
    \left.\begin{aligned}
\pddt x(t,\xi)=&\,\divg a(\xi) \grad x(t,\xi)\\&\,+b(\xi)^\top \grad x(t,\xi)+c(\xi)\,x(t,\xi),&&(t,\xi)\in[0,\infty)\times \Omega\\[1mm]
\nu^\top a(\xi)\grad x(t,\xi)=&\sum_{i=1}^mu_i(t)w_i(\xi), &&(t,\xi)\in[0,\infty)\times \Gamma,\\ 
\int_{\Gamma}\overline{w_i(\xi)}x(t,\xi)\mathrm{d}\sigma=&\,y_i(t), &&t\in[0,\infty),\\[-2mm]&&& \,i=1,\ldots,m.
    \end{aligned}\right\}
\end{equation}
Hereby, we impose some assumptions on the domain and the involved functions:
\begin{ass}\label{ass:parab}
\begin{enumerate}[(i)]\
\item For $n\in\N$, $\Omega\subset\R^n$ is a bounded domain with Lipschitz boundary $\Gamma$;
\item $a\in L^\infty(\Omega;\K^{n\times n})$ has values in the set of positive definite Hermitian matrices and $\zeta\mapsto a(\zeta)^{-1}$ is essentially bounded as well;
\item $b\in L^\infty(\Omega;\K^{n})$, $c\in L^\infty(\Omega;\K)$;
\item $m\in\N$, and $w_1,\ldots,w_m\in W^{-1/2,2}(\Gamma):=W^{1/2,2}(\Gamma)'$ are linearly independent.
\end{enumerate}
\end{ass}
In the following, we construct a BCS~$\syst$ from the model problem~\eqref{eq:str_parab}. 
Loosely speaking, the operator $\Af$ is given by
$x\mapsto\divg a \grad x+b^\top \grad x+cx$, whereas $\Bf$ and $\Cf$ respectively correspond to an evaluation of the Neumann boundary trace and the weighted integral of the Dirichlet boundary trace. We further show that Green's identity implies that the resulting BCS fulfills the generalized impedance passivity condition \eqref{eq:diss}.\\
To construct the operators in a~rigorous way, we use the space of functions with square integrable weak divergence, i.e.,
$$H_{\rm div}(\Omega):=\{x\in L^2(\Omega;\K^n)\st \divg x\in L^2(\Omega) \},$$
which becomes a~Hilbert space in a~natural way. We collect some properties of $H_{\rm div}(\Omega)$ and traces, which are used for the construction of $\syst$ and to verify that this BCS fulfills Assumptions~\ref{ass:system_class}.
To this end, note that $W^{-1/2,2}(\Gamma)$ denotes the dual space of $W^{1/2,2}(\Gamma)$ by extending the inner product in $L^{2}(\Gamma)$.

\begin{proposition}\label{eq:hdiv}
Let $\Omega\subset\R^n$, $n\in \N$, be a bounded domain with Lipschitz boundary~$\Gamma$.
\begin{enumerate}[(a)]
    \item\label{eq:hdiv1} The operator $\widetilde{\cT}_{0}: C^\infty(\overline{\Omega})\to W^{1,\infty}(\Gamma)$ with $x\mapsto \left. x\right|_{\Gamma}$ extends to a~bounded and surjective operator ${\cT}_{0}: W^{1,2}(\Omega)\to W^{1/2,2}(\Gamma)$ \cite[Theorem 1.3]{GiraultRaviart}.
    \item\label{eq:hdiv2} The operator $\widetilde{\cT}_{0,\nu}: C^\infty(\overline{\Omega};\R^n)\to L^\infty(\Gamma)$ with $x\mapsto \left. x\right|_{\Gamma}$ extends to a~bounded and surjective operator ${\cT}_{0,\nu}: H_{\rm div}(\Omega)\to W^{-1/2,2}(\Gamma)$ \cite[Thm.~2.2~Cor.~2.4]{GiraultRaviart}.
    \item\label{eq:hdiv3} For all $x\in  H_{\rm div}(\Omega)$, $z\in  W^{1,2}(\Omega)$ it holds that
    \[\scpr{\divg x}{z}_{L^2(\Omega)}+\scpr{x}{\grad z}_{L^2(\Omega;\R^n)}=\scpr{{\cT}_{0,\nu}x}{{\cT}_{0}z}_{W^{-1/2,2}(\Gamma),W^{1/2,2}(\Gamma)},\]
    see \cite[Cor.~2.1]{GiraultRaviart}.
\end{enumerate}
\end{proposition}
We define the space
\begin{equation*}
    \cD(\Af)=\setdef{x\in W^{1,2}(\Omega)}{\begin{array}{l}a\cdot\grad x\in H_{\rm div}(\Omega)\;\;\wedge\\\; \cT_{0,\nu}(a \grad x)\in{\rm span}_{\K}(w_1,\ldots,w_m)\!\!\!\!\!\end{array}},
\end{equation*}
equipped with the norm 
\[\|x\|_{\cD(\Af)}=\left(\|x\|_{L^{2}(\Omega)}^{2}+\|a\grad x\|_{L^{2}(\Omega;\K^{n})}^{2}+\|\divg a\grad x\|_{L^{2}(\Omega)}^{2}\right)^{1/2}\]
and the operator
\begin{equation}
    \Af:\quad\cD(\Af)\to L^2(\Omega),x\mapsto \divg a \grad x+b^\top \grad x+cx.
    \label{eq:Afparab}
\end{equation}
Furthermore, we define operators $\Bf, \Cf:\cD(\Af)\to\K^m$  by
\begin{equation}\Bf x=(\lambda_1,\ldots,\lambda_m),\text{ where }\cT_{0,\nu}(a \grad x)=\lambda_1 w_1+\ldots,\lambda_mw_m,\label{eq:Bfparab}
\end{equation}
and
\begin{equation}\Cf x=\begin{pmatrix}\scpr{\cT_{0}x}{w_1}_{W^{1/2,2}(\Omega),W^{-1/2,2}(\Omega)}\\\vdots\\\scpr{\cT_{0}x}{w_m}_{W^{1/2,2}(\Omega),W^{-1/2,2}(\Omega)}\end{pmatrix}.\label{eq:Cfparab}
\end{equation}
 By Proposition \ref{eq:hdiv}, $\Af, \Bf$ and $\Cf$ are well-defined and continuous with respect to $\|\cdot\|_{\cD(\Af)}$. Moreover,  $\cD(\Af)$ is complete.
In the remaining part of this section we show that the BCS $\syst$ satisfies Assumptions~\ref{ass:system_class}. 
A convenient tool will be the~sesquilinear form $\fa_\beta:W^{1,2}(\Omega)\times W^{1,2}(\Omega)\to\,\K$, for $\beta\in\R$, defined by
\begin{equation}
\begin{aligned}
\fa_\beta
(x_1,x_2)=\scpr{a\grad x_1}{\grad x_2}_{L^2(\Omega;\R^n)}-\scpr{ b^\top\grad x_1-(\beta-c)x_1}{ x_2}_{L^2(\Omega)}.\label{eq:sesq}
\end{aligned}\end{equation}
It can immediately be verified that $\fa_\beta$ satisfies the following G\r{a}rding-type inequality,
\begin{multline} \re\left(\fa_\beta(x,x)\right)\geq C_1\, \|x\|^2_{W^{1,2}(\Omega)}+(\beta-C_2)\ \|x\|^2_{L^{2}(\Omega)}\quad\forall x\in W^{1,2}(\Omega), \beta\in\R,\label{eq:parab_sesqest}\end{multline}
for some constants $C_{1},C_{2}>0$ independent of $\beta$ and $x$.
Viewing a parabolic problem such as  \eqref{eq:str_parab} by means of forms is a well-known strategy, see e.g.\ \cite{ArenElst12,Arendtetal}. 
\begin{lemma}\label{lem:Aop}
Let Assumption~\ref{ass:parab} be valid, $\beta\in\R$, and let $\fa_\beta$ be the sesquilinear form defined by \eqref{eq:sesq}. 
\begin{enumerate}[(a)]
\item
Then $x\in\cD(\Af)$ if, and only if, there exists some $z\in L^2(\Omega)$ and $\lambda_1,\ldots,\lambda_m\in\K$, such that for all $x_{2}\in W^{1,2}(\Omega)$ it holds that
\begin{equation}
\fa_\beta(x,x_{2})=-\scpr{z}{x_{2}}_{L^2(\Omega)\!\!}+\sum_{k=1}^m\lambda_k\scpr{w_k}{\cT_0 x_{2}}_{W^{-1/2}(\Gamma),W^{1/2}(\Gamma)}.\label{eq:sesqformsys}    
\end{equation}
Further, the elements $z\in L^2(\Omega)$ and $\lambda_1,\ldots,\lambda_m\in\K$ with the above property are uniquely determined by $x\in\cD(\Af)$, and it holds that $(\Af-\beta I) x=z$ and $\Bf x=(\lambda_1,\ldots,\lambda_m)$.
\item The space $\cD(\Af)\cap\ker \Bf$ is dense in $W^{1,2}(\Omega)$.
\end{enumerate}
\end{lemma}
\begin{proof}
\begin{enumerate}[(a)]
\item As we can consider $c-\beta$ instead of $c$, it suffices to prove the statement for $\beta=0$.
If $x\in\cD(\Af)$, then Green's identity (see Proposition~\ref{eq:hdiv}\eqref{eq:hdiv3}) implies that for all $x_2\in W^{1,2}(\Omega)$, we have that
\[
\begin{aligned}
\scpr{\Af x}{x_2}_{L^2(\Omega)}=&\,\scpr{\divg a \grad x}{x_2}_{L^2(\Omega)}+
\scpr{b^\top \grad x+cx}{x_2}_{L^2(\Omega)}\\
=&\,-\scpr{a \grad x}{\grad x_2}_{L^2(\Omega;\R^n)}+
\scpr{b^\top \grad x+cx}{x_2}_{L^2(\Omega)}\\
&\quad+\scpr{{\cT}_{0,\nu}(a \grad x)}{{\cT}_{0}x_2}_{W^{-1/2,2}(\Gamma),W^{1/2,2}(\Gamma)}\\
=&\,-\fa_0(x,x_2)+\sum_{k=1}^m\lambda_{k}\scpr{ w_k}{{\cT}_{0}x_2}_{W^{-1/2,2}(\Gamma),W^{1/2,2}(\Gamma)}.
\end{aligned}
\]
On the other hand, suppose that $z\in L^2(\Omega)$ and $\lambda_1,\ldots,\lambda_m\in\K$ such that
\eqref{eq:sesqformsys} holds for all $x_2\in W^{1,2}(\Omega)$. In particular, for all $\varphi\in C^\infty_0(\Omega)$ it holds that
\[-\scpr{a \grad x}{\grad \varphi}_{L^2(\Omega)}=\scpr{z}{ \varphi}_{L^2(\Omega)}-\scpr{b^\top \grad x+cx}{\varphi}_{L^2(\Omega)}.\]
Consequently, $(a\grad x)\in H_{\rm div}(\Omega)$ with 
$\divg a \grad x=z-b^\top \grad x+cx$.
Using Green's identity (Proposition~\ref{eq:hdiv}\eqref{eq:hdiv3}), we thus have for all $x_2\in W^{1,2}(\Omega)$ that
\[\begin{aligned}
&\sum_{k=1}^m\lambda_{k}\scpr{ w_k}{{\cT}_{0}x_2}_{W^{-1/2,2}(\Gamma),W^{1/2,2}(\Gamma)}\\=&\;\fa_0(x,x_2)+\scpr{z}{ x_2}_{L^2(\Omega)}=\scpr{{\cT}_{0,\nu}a\grad x}{{\cT}_{0}x_2}_{W^{-1/2,2}(\Gamma),W^{1/2,2}(\Gamma)}.
\end{aligned}\]
As ${\cT}_{0}:W^{1,2}(\Omega)\to W^{1/2,2}(\Gamma)$ is onto by  Proposition~\ref{eq:hdiv}\eqref{eq:hdiv1}, we obtain that 
\[{\cT}_{0,\nu}a\grad x=\sum_{k=1}^m\lambda_k w_k.\]
Hence, $x\in\cD(\Af)$ with $z=\Af x$ and  $\Bf x=(\lambda_1,\ldots,\lambda_k)$, which completes the proof of the equivalence. The uniqueness follows since $\Af$ is  well-defined.
\item
By (a), $\cD(\Af)\cap \ker\Bf$ equals the space of all $x\in W^{1,2}(\Omega)$ for which there exists some $z\in L^2(\Omega)$ such that $\fa_0(x,x_2)=-\scpr{z}{x_2}_{L^2(\Omega)}$ for all $x_2\in W^{1,2}(\Omega)$. 
For $\K=\C$, we apply  \cite[Thm.~5.9]{Arendtetal} to obtain that Estimate \eqref{eq:parab_sesqest} yields that the sesquilinear form $\fa_0$ is sectorial in the sense of the definition in \cite[p.~310]{Kato80}. 
Then Kato's first representation theorem \cite[Sec.~VI.2, Thm.~2.1]{Kato80} implies that $\cD(\Af)\cap \ker\Bf$
is dense in $W^{1,2}(\Omega)$. The case $\K=\R$ then simply follows by considering $\fa_0$ as a~sesquilinear form on the complex space $W^{1,2}(\Omega)$.
\end{enumerate}
\end{proof}

Next we show that $\syst$ satisfies the requirements of Theorem~\ref{thm:funnel2}.
\begin{theorem}\label{thm:funnel_heat}
Let Assumption~\ref{ass:parab} be valid, let the operators $\Af$, $\Bf$ and $\Cf$  be defined as in \eqref{eq:Afparab},
\eqref{eq:Bfparab} and \eqref{eq:Cfparab}.
Then the~BCS $\syst$ fulfills Assumptions~\ref{ass:system_class}. If additionally, $b=0$ and the real part of $c$ is essentially bounded from above by a~negative constant, then  Assumption~\ref{ass:system_class} is satisfied with $\alpha<0$.
\end{theorem}
\begin{proof}
{\em Step 1:} 
We show that there exists some $\alpha>0$  such that \eqref{eq:diss} is satisfied. Let $x\in\cD(\Af)$. 
Using Lemma~\ref{lem:Aop}(a), we have for all $x\in\cD(\Af)$ that 
\[\begin{aligned}
\re\scpr{\Af x}{x}=&\,-\re\fa_0(x,x)+\sum_{k=1}^m\lambda_k\scpr{w_k}{\cT_0 x}_{W^{-1/2}(\Gamma),W^{1/2}(\Gamma)}\\
=&-\re\fa_0(x,x)+\scpr{\Bf x}{\Cf x}_{\K^m}.
\end{aligned}\]
Now using Estimate \eqref{eq:parab_sesqest}, we obtain that \eqref{eq:diss}  holds with $\alpha=C_2$.\\
{\em Step 2:} We prove that $\alpha$ in \eqref{eq:diss} can be chosen negative, if $b=0$ and $\re(c)$ is essentially bounded from above by some positive constant. This follows analogously as in Step~1 using that, in the case $b=0$, we have that
\[ \re\left(\fa_0(x,x)\right)\geq -\esssup_{\xi\in\Omega}\re(c(\xi))\, \|x\|^2_{L^{2}}\quad\forall x\in W^{1,2}(\Omega).\]
{\em Step 3:} We show that  $\Af|_{\ker\Cf}-\beta I$ maps onto $L^2(\Omega)$ for $\beta>\alpha=C_2$. To this end, let $z\in L^2(\Omega)$.
By using \eqref{eq:parab_sesqest}, we see that the sesquilinear form $\fa_\beta$
is coercive and bounded, whence we can infer from the Lax-Milgram theorem, see e.g.\ \cite[Thm.~6.2]{Alt16}, that there exists a unique $x\in W^{1,2}(\Omega)$ with 
\[    \fa_\beta(x,v)=-\left\langle z,v\right\rangle_{L^2(\Omega)}\quad \forall v\in W^{1,2}(\Omega).\]
Then Lemma~\ref{lem:Aop}(a) implies that $x\in \cD(\Af)$ with $(\Af-\beta I) x=z$ and $\Bf x=0$, i.e., 
$(\left.\Af\right|_{\ker\Bf}-\beta I) x=z$.
By using that the already proven inequality \eqref{eq:diss} holds for some $\alpha\in\R$ and our input and output spaces are finite-dimensional, we can use Proposition~\ref{prop:ass} to conclude that $\cR(\Af|_{\ker\Cf}-\beta I)=L^2(\Omega)$ 
for $\beta\ge\alpha$.\\
{\em Step 4:} We prove that $\left[\begin{smallmatrix}\Bf\\\Cf\end{smallmatrix}\right]$ maps $\cD(\Af)$ onto $\K^{2m}$. This is done in three steps.  \\ 
{\em Step 4a:} We show that $\cR(\Bf)=\K^m$.  Let $u\in \K^{n}$ and $\beta>C_1$, where $C_{1}$ was defined in \eqref{eq:parab_sesqest}.
Again by using the Lax-Milgram theorem we find some $x\in W^{1,2}(\Omega)$ with 
\[\fa_\beta(x,x_2)=\left\langle u,\left(\langle \cT_{0} x_2,w_i \rangle\right)_{i=1,\ldots,m}\right\rangle_{\K^m}\quad \forall x_2\in W^{1,2}(\Omega).\]
Lemma~\ref{lem:Aop}(a) thus yields $x\in\cD(\Af)$, $(\Af-\beta I)x=0$ and $\Bf x=u$, whence $u\in\cR(\Bf)$.\\
{\em Step 4b:} 
We show that $\Cf(\cD(\Af)\cap\ker\Bf)=\K^{m}$. 
Since $w_1,\ldots,w_m\in W^{-1/2,2}(\Gamma)$ are linearly independent, 
the operator 
\[P:W^{1/2,2}(\Gamma)\to \K^{m},w\mapsto(\langle w,w_{i}\rangle)_{i=1,..,m}\]
is bounded and surjective and so is $P\circ \cT_{0}:W^{1,2}(\Omega)\to\K^{m}$, by Proposition~\ref{eq:hdiv}\eqref{eq:hdiv1}. Clearly, $P\circ\cT_{0}$ extends the operator $\Cf$. Since $\cD(\Af)\cap\ker\Bf$ is dense in $W^{1,2}(\Omega)$ by Lemma~\ref{lem:Aop}(b), we conclude that $\Cf(\cD(\Af)\cap\ker\Bf)$ is dense and thus equals $\K^{m}$.\\
{\em Step 4c:} We conclude that $\cR\left(\left[\begin{smallmatrix}\Bf\\\Cf\end{smallmatrix}\right]\right)=\K^{2m}$. Let $u,y\in\K^m$. By Step~4a, there exists $x_1\in\cD(\Af)$ with $u=\Bf x_1$. Further, by Step~4b, we find $x_2\in\cD(\Af)$ with $\Bf x_2=0$ and $\Cf x_2=y-\Cf x_1$. Then $\left[\begin{smallmatrix}\Bf\\\Cf\end{smallmatrix}\right](x_1+x_2)=(\begin{smallmatrix}u\\y\end{smallmatrix})$
and thus $\left(\begin{smallmatrix}u\\y\end{smallmatrix}\right)\in \cR\left(\left[\begin{smallmatrix}\Bf\\\Cf\end{smallmatrix}\right]\right)$.
\end{proof}
{\begin{remark}
\begin{enumerate}[(a)]
\item Funnel control for systems of type \eqref{eq:str_parab} with, additionally, $m=1$, $w_1\equiv 1$, $a\equiv I_m$, $b\equiv0$, $c\equiv0$
and boundary $\Gamma$ being of class $C^2$ has been treated in \cite{ReisSelig}. Our result allows, besides the treatment of less smooth boundary and more general parabolic equations, also to have multiple actuators and sensors. A~typical situation may be functions $w_1,\ldots,w_m$ which are supported at several disjoint parts of the boundary.
\item 
Under the additional assumptions that the boundary $\Gamma$ of $\Omega$ is of class $C^{1,1}$, $w_1,\ldots,w_m\in W^{1/2,2}(\Omega)$ and that 
$a$ is uniformly Lipschitz, it holds that
 \begin{equation}\cD(\Af)=\setdef{x\in W^{2,2}(\Omega)}{\cT_{0,\nu} (a\grad x)\in{\rm span}_{\K}(w_1,\ldots,w_m)}.\label{eq:DAfrule} \end{equation}
Indeed, $\Af-\beta I$ satisfies the assumptions of \cite[Thm.~2.4.7]{Grisvard} for sufficiently large $\beta$.
This implies that the equation $(\Af-\beta I)x=z$ has a unique solution $x\in W^{2,2}(\Omega)$ for any given $z\in L^{2}(\Omega)$. Hence $\cD(\Af)\subset W^{2,2}(\Omega)$ by Lemma~\ref{lem:Aop}(a).
\item Since $w_1,\ldots,w_m\in W^{-1/2,2}(\Gamma)$ are linear independent, there exists some dual system 
$v_1,\ldots,v_m\in W^{1/2,2}(\Gamma)$ with $\langle 
w_i,v_j\rangle=\delta_{ij}$. Then $\Bf$ in 
\eqref{eq:Bfparab} can be reformulated to
\[\Bf x=(\langle \cT_{0} x,v_1\rangle,\ldots,\langle \cT_{0} x,v_m\rangle).\]
\end{enumerate}
\end{remark}
}\color{black}

\section{Proof of Proposition~\ref{prop:ass} and Theorems \ref{thm:funnel1} \& \ref{thm:funnel2}}
\label{sec:proof}

We develop some auxiliary results to facilitate the proofs of the main results.
A~part of following lemma has been shown in \cite{ChengMorris} under the additional assumption of well-posedness, cf.~Remark~\ref{rem:system_class}\eqref{item:wpls}.

{\begin{lemma}\label{lemm:resolvent_gen}
Assume that $\syst$ satisfies {Assumptions~\ref{ass:system_class}\eqref{ass:1}\&\eqref{ass:3}\&\eqref{ass:CBonto} with $\alpha\in\R$.}
	For all $\beta>\alpha$, ${y\in U}$ and $f\in X$ there exist unique $x\in\cD(\Af)$ and {$u\in{U}$} with
	\begin{equation}\label{eq:syst_laplace}
	(\beta I-\Af)x=f,\quad	u=\Bf x,\quad
	y=\Cf x.
	\end{equation}
	Furthermore, there exist operators $H(\beta)\in\cL(X)$, $J(\beta)\in\cL({U},X)$, $F(\beta)\in\cL(X,{U})$ and $G(\beta)\in\cL({U})$, such that $x$, $f$, $u$ and $y$ fulfill \eqref{eq:syst_laplace} if, and only if,
	\begin{equation}\label{eq:system_sol}
	\begin{aligned}
	x&=H(\beta)f+J(\beta)y,\\
	u&=F(\beta)f+G(\beta)y.
	\end{aligned}
	\end{equation}
Further, the operator $G(\beta)$ fulfills  $\re\langle y, G(\beta)y\rangle_U>0$ for all $y\in U\setminus\{0\}$.
\end{lemma}}

\begin{proof}
{{\em Step 1:} Provided that a~solution of \eqref{eq:syst_laplace} exists, we show it is unique.
To this end, it suffices to show that the choice $f=0$ and {$y=0$} leads to $x=0$ and {$u=0$}. Assuming that $x\in\mathcal{D}(\Af)$, $y\in{U}$ fulfills  \eqref{eq:syst_laplace} with $f=0$ and {$y=0$}, we obtain from
\eqref{eq:diss} in Assumption~\ref{ass:system_class}\eqref{ass:passivity} that
	 \[\beta\|x\|_X^2	=\re\langle \Af x, x\rangle_X\leq\re\langle \Bf x,\Cf x\rangle_{{U}}+\alpha\|x\|_X^2,\]
and thus $(\beta-\alpha)\|x\|^2\leq0$. Invoking $\beta>\alpha$, we obtain $x=0$ and, consequently, {$u=\Cf x=0$}.\\}
{{\em Step 2:} We show the existence of bounded operators $H(\beta)$, $J(\beta)$, $F(\beta)$ and $G(\beta)$
such that $x$ and $u$ given by \eqref{eq:system_sol}  satisfy \eqref{eq:syst_laplace}.
By using Remark~\ref{rem:system_class}\eqref{rem:system_class2}, we see that Assumptions~\ref{ass:system_class}\eqref{ass:passivity}\&\eqref{ass:3} imply that $\beta I-\Af|_{\ker\Cf}$ is bijective.
Further invoking Remark~\ref{rem:system_class}\eqref{rem:system_class3}, Assumption~\ref{ass:system_class}\eqref{ass:CBonto} leads to the existence of
$P,Q\in\cL({U},\cD(\Af))$, such that \eqref{eq:onto} holds. Considering
	\begin{align*}
	x&=\underbrace{(\beta I-\Af|_{\ker\Cf})^{-1}}_{=:H(\beta)}f+\underbrace{((\beta I-\Af|_{\ker\Cf})^{-1}(\Af Q-\beta Q)+Q)}_{=:J(\beta)}y,\\
	u&=\underbrace{\Bf(\beta I-\Af|_{\ker\Cf})^{-1}}_{=:F(\beta)}f+\underbrace{\Bf(\beta I-\Af|_{\ker\Cf})^{-1}(\Af Q-\beta Q)}_{=:G(\beta)}y,
	\end{align*}
a~straightforward calculation shows that \eqref{eq:syst_laplace} holds. Further, the operators $H(\beta)$, $J(\beta)$, $F(\beta)$ and $G(\beta)$ are bounded as  compositions of bounded operators. Together with Step 1 this also shows that a quadruple $(x,u,y,f)$ solves \eqref{eq:syst_laplace} if, and only if, \eqref{eq:system_sol} holds. \\
{\em Step 3:} We show that $\re\langle y, G(\beta)y\rangle_U>0$ for all $y\in U\setminus\{0\}$.
Considering (\ref{eq:syst_laplace}) with $f=0$ and taking the real part of the inner product in $X$, we obtain
	$$\re\beta\|x\|^2=\re\langle\Af x,x\rangle\leq\re\langle u,y\rangle_{{U}}+\alpha\|x\|^2=\re\langle G(\beta)y,y\rangle_{{U}}\!+\alpha\|x\|^2,$$
	whence
	$$0\leq(\beta-\alpha)\|x\|^2\leq\re\langle G(\beta)y,y\rangle_{{U}},$$
	so that $G(\beta)$ is accretive. To prove that this inequality is strict for $y\neq0$, assume that $y\in{U}$ fulfills $\re\langle G(\beta)y,y\rangle_{{U}}=0$. Then $(\beta-\alpha)\|x\|^2_X=0$, which gives $x=0$, and thus $y=\Cf x=0$.
}\end{proof}
{Now we are able to present the proof of Proposition~\ref{prop:ass}.
\begin{proof}[Proof of Proposition~\ref{prop:ass}]
Assume that $\mathfrak{S}=(\Af,\Bf,\Cf)$ is a~BCS satisfying Assumptions~\ref{ass:system_class}\eqref{ass:1}\&\eqref{ass:CBonto}.\\
As already stated in Remark~\ref{rem:system_class}~\eqref{rem:system_class1}, the equivalence between (i) and (ii) as well as the one between (iii) and (iv) are direct consequences of the Lumer–Phillips theorem \cite[Thm.~3.15]{EngeNage00}. It remains to prove the equivalence between (i) and (iii) under the additional assumption that $\dim U<\infty$.\\ 
``(i)$\Rightarrow$(iii)'': Assume that $\dim U<\infty$ and that Assumption~\ref{ass:system_class}\eqref{ass:3} holds.
Let $\beta>\alpha$ and $f\in X$. Consider $H(\beta)\in\cL(X)$, $J(\beta)\in\cL({U},X)$, $F(\beta)\in\cL(X,{U})$ and $G(\beta)\in\cL({U})$ as in Lemma~\ref{lemm:resolvent_gen}. Since $\dim U<\infty$ and $\re\langle y, G(\beta)y\rangle_U>0$ for all $y\in U$, $G(\beta)$ has a~bounded inverse. Hence we may define
$x:=(H(\beta)-J(\beta)G(\beta)^{-1}F(\beta))f$. Then for $y:=-G(\beta)^{-1}F(\beta)f$ it holds that
\[	\begin{aligned}
	x&=H(\beta)f+J(\beta)y,\\
	0&=F(\beta)f+G(\beta)y.
	\end{aligned}\]
Therefore Lemma~\ref{lemm:resolvent_gen} implies  in particular that $(\beta I-\Af)x=f$ and $0=\Bf x$, whence $(\Af|_{\ker\Bf}-\beta I)x=f$.\\
 ``(iii)$\Rightarrow$(i)'': If $\dim U<\infty$ and $\cR(\Af|_{\ker\Bf}-\beta I)=X$ for some $\beta\ge\alpha$, then by interchanging the roles of $\Bf$ and $\Cf$, we see that the BCS $(\Af,\Cf,\Bf)$ fulfills Assumption~\ref{ass:system_class}\eqref{ass:1},\eqref{ass:3}\&\eqref{ass:CBonto}. As already proven in ``(i)$\Rightarrow$(iii)'', this implies that $\cR(\Af|_{\ker\Cf}-\beta I)=X$ for some $\beta>\alpha$, which completes the proof.
\end{proof}}



{Let us introduce the following classical notion in the operator-theoretic study of nonlinear evolution equations, see e.g.\ \cite{Miyadera,Showalter}.}
\begin{definition}
	Let $X$ be a Hilbert space and $A:\cD(A)\subset X\rightarrow X$ a (possibly nonlinear) operator. We say that $A$ is \emph{dissipative}, if for all $x,y\in\cD(A)$, the inequality $\re\langle A(x)-A(y),x-y\rangle\leq0$ holds. If, furthermore,
	for all $\lambda>0$, it holds that $\cR(\lambda I-A)=X$, we call $A$ \emph{m-dissipative}.
\end{definition}

\begin{remark}
	If $A:\cD(A)\subset X\to X$ is m-dissipative, then for all $f\in X$ and $\lambda>0$ there exists some $z\in\cD(A)$ with
	$\lambda z-A(z)=f$.
	The element $z$ is indeed unique, since for any $x\in\cD(A)$ with $\lambda x-A(x)=f$, we obtain by taking the difference that
	$$\lambda(x-z)-(A(x)-A(z))=0$$ and taking the inner product with $x-z$ gives
	$$\lambda\|x-z\|^2=\re\langle A(x)-A(z),x-z\rangle.$$
	Dissipativity of $A$ leads to non-positivity of the latter expression, whence $x=z$.
\end{remark}

\begin{lemma}\label{lemma:funnel_m_monotone}
	Let $\phi:B_{1}(0)\subset{U}\rightarrow{U}$ be defined by
	\begin{equation}\label{eq:funnel_function}
	\phi(y):=\dfrac{1}{1-\|y\|^2}y.
	\end{equation}
	Then $-\phi$ is {dissipative.}
\end{lemma}

\begin{proof}
{Consider} the function $g:[0,1)\to\R$ defined by
	$r\mapsto \tfrac{r}{1-r^2}$
is monotonically increasing on $[0,1)$, which follows from nonnegativity of its derivative. As a~consequence $(g(a)-g(b))(a-b)\geq0$ for all $a,b\in[0,1)$. Using this, we obtain that for $w,y\in B_{1}(0)$,
	\begin{align*}
	\re\langle\phi(w)-\phi(y),w-y\rangle&=\re\langle\phi(w),w\rangle+\re\langle\phi(y),y\rangle-\re\langle \phi(y),w\rangle-\re\langle\phi(w),y\rangle\\
	&=\left(\tfrac{\|w\|^2}{1-\|w\|^2}+\tfrac{\|y\|^2}{1-\|y\|^2}-\tfrac{\re\langle w,y\rangle}{1-\|y\|^2}-\tfrac{\re\langle y,w\rangle}{1-\|w\|^2}\right)\\
	&\geq \left(\tfrac{\|w\|^2}{1-\|w\|^2}+\tfrac{\|y\|^2}{1-\|y\|^2}-\tfrac{\|w\|\|y\|}{1-\|y\|^2}-\tfrac{\|y\|\|w\|}{1-\|w\|^2}\right)\\
	&=\left(\tfrac{\|w\|}{1-\|w\|^2}-\tfrac{\|y\|}{1-\|y\|^2}\right)(\|w\|-\|y\|)\\
&= (g(\|w\|)-g(\|y\|))\,(\|w\|-\|y\|)\geq0.
	\end{align*}
\vspace{-1.1cm}

%
%
\end{proof}

%

The next result is a modification of \cite[Thm.~4.3]{AugnerDis} in which the function $\phi$ is defined on the whole space ${U}$ instead of the domain $B_{1}(0)$, as in our situation.

\begin{lemma}\label{lem:A_m_diss}
	Let $\syst$ be a BCS and let Assumption \ref{ass:system_class} be satisfied with $\alpha\in\R$. Let $\phi:B_{1}(0)\subset{U}\rightarrow{U}$ be given by \eqref{eq:funnel_function}. Then the nonlinear operator $\cA:D(\cA)\subset X\to X$ with
	\begin{equation}\label{eq:closed_loop_op}
	\begin{aligned}
	\cA&:=(\Af-\alpha I)|_{\cD(\cA)},\\
	\cD(\cA)&:=\{z\in\cD(\Af)\st\|\Cf z\|<1,\Bf z+\phi(\Cf z)=0\}
	\end{aligned}
	\end{equation}
	is m-dissipative. {If $\ker\Bf\cap\ker\Cf$ is dense, then ${\cD(\cA)}$ is dense in $X$.}
\end{lemma}

\begin{proof}
{\em Step 1:}
	{Since $\ker\Bf\cap\ker\Cf$ is contained in $D(\cA)$, it follows that $\cA$ is densely defined if $\ker\Bf\cap\ker\Cf$  dense in $X$.}\\
{{\em Step 2:} Let $\beta>\alpha$ and let $G(\beta)\in \cL(U)$ be as in Lemma~\ref{lemm:resolvent_gen}. 
We show that the function $\Psi:B_1(0)\to U$ with	
\begin{equation}\Psi(\cdot):= \phi(\cdot)+G(\beta)\label{eq:Psidef}\end{equation}
is onto. 
Let $u\in U$. We aim to show that $u=\Psi(e)$ for some $e\in B_1(0)$. If $u=0$, then this is clearly fulfilled by setting $e=0$. Now assume that $u\neq0$.  The dissipativity of $-G(\beta)$ (which holds by Lemma~\ref{lemm:resolvent_gen}) gives rise to bounded invertibility of $\xi I+G(\beta)$ for all $\xi>0$. Hence, we may consider the
continuous function 
\[\begin{aligned}&p:&[1,\infty)&\,\to\R,\\&&\xi&\,\mapsto\|(\xi I+G(\beta))^{-1}u\|^2-1+\xi^{-1}.\end{aligned}\]
Then $p(1)>0$ and $\lim_{\xi\rightarrow\infty} p(\xi)=-1$. Hence there exists some $\xi^\star\in(1,\infty)$ with $p(\xi^\star)=0$.
Let $e:=(\xi^\star I+G(\beta))^{-1}u$. Then $p(\xi^\star)=0$ implies that $\|e\|^2=1-(\xi^\star)^{-1}$. This gives rise to $e\in B_1(0)$ and
\[\phi(e)=\frac{1}{1-(1-(\xi^\star)^{-1})}\, e=\xi^\star\, e=\xi^\star (\xi^\star I+G(\beta))^{-1}u,\]
which leads to 
\[\begin{aligned}
\Psi(e)=&\,\phi(e)+G(\beta)e\\=&\,\xi^\star (\xi^\star I+G(\beta))^{-1}u+G(\beta) (\xi^\star I+G(\beta))^{-1}u\\
=&\,(\xi^\star I+G(\beta))(\xi^\star I+G(\beta))^{-1}u=u.\end{aligned}
\]
{\em Step 3:} We show that $\lambda I-\cA$ is onto for all $\lambda>0$:\\
	Let $f\in X$. Our aim is to find some $z\in\cD(\cA)$ with $(\lambda I-\cA)(z)=f$, that is, $\|\Cf z\|<1$ and
\begin{equation}\begin{aligned}
	((\lambda+\alpha) I-\Af)z&=f\\
	\Bf z&=-\phi(\Cf z).
\end{aligned}\label{eq:Afrel}\end{equation}
Set $\beta:=\lambda+\alpha>\alpha$ and consider the operators $H(\beta)\in\cL(X)$, $J(\beta)\in\cL({U},X)$, $F(\beta)\in\cL(X,{U})$ and $G(\beta)\in\cL({U})$ from Lemma~\ref{lemm:resolvent_gen}.
Then $\Psi$ as in \eqref{eq:Psidef} is onto by Step~2, whence there exists some $e\in  U$ with $\|e\|<1$ and $\Psi(e)=-F(\beta)f$. The latter is equivalent to
\[-\phi(e)=F(\beta)f+G(\beta)e,\]
and Lemma~\ref{lemm:resolvent_gen} implies that $z=H(\beta)f+J(\beta)e$
indeed fulfills \eqref{eq:Afrel} together with $\|\Cf z\|<1$.} \\
{\em Step~4:} We show that $\cA$ is dissipative: Let $z_1,z_2\in\cD(\cA)$, then
	\[\begin{array}{rcl}
	\re\langle\cA(z_1)-\cA(z_2),z_1-z_2\rangle_X
&=&
\re\langle(\Af-\alpha I) z_1-(\Af-\alpha I) z_2,z_1-z_2\rangle_X\\[3mm]
	&\underset{\ref{ass:system_class}\eqref{ass:passivity}}{\overset{\rm Assumpt.}{\leq}}&-\re\langle\phi(\Cf z_1)-\phi(\Cf z_2),\Cf z_1- \Cf z_2\rangle_{U}\\[2mm]
	&\underset{\ref{lemma:funnel_m_monotone}}{\overset{\rm Lemma}{\leq}} &0.\\[-6mm]
	\end{array}\]
\end{proof}
%

An intrinsic technical problem when investigating solvability of  \eqref{eq:funnel_problem} is that the feedback is varying in time, i.e.\ it depends on $t$ explicitly.
To circumvent this problem, we perform a change of variables leading to an evolution equation with a constant operator. This is the content of the subsequent auxiliary result.
{
\begin{proposition}\label{lem:transf}
Let $\syst$ be a~BCS satisfying Assumption~\ref{ass:system_class} for some $\alpha\in\R$ and assume that the initial value $x_0$ and the functions $\yref$, $p$, $\varphi$ fulfill Assumption~\ref{ass:contr}\tb{, and let the operators $P,Q\in\cL({U},\cD(\Af))$ be as in \eqref{eq:onto}. Then for the m-dissipative operator $\cA$ given in (\ref{eq:closed_loop_op}), $$\varphi_0:=\varphi(0),\quad e_0:=\Cf x_0-\yref(0),\quad u_0:=\Bf x_0,\quad 	\psi(\varphi,e):=\dfrac{1}{1-\varphi^2\,\|e\|^2}\, e,$$ and
 	\begin{equation}\label{eq:data}
	\left.\begin{aligned}
	\omega:={}&\tfrac{\dot{\varphi}}{\varphi},\\
	f:={}&\varphi \,\Big(\Af Q\yref-Q\dot{y}_{\rm ref}+\Af P(u_0+\psi(\varphi_0,e_0))p -P(u_0+\psi(\varphi_0,e_0))\dot{p}\Big),\\
	z_0:={}&\varphi_0\,\Big(x_0-Q\yref(0)-P(u_0+\psi(\varphi_0,e_0))\Big) \in\mathcal{D}(\cA)
	\end{aligned}\right\}
	\end{equation}
it holds that
\[\omega\in W^{1,\infty}([0,\infty)),\quad f\in  W^{1,\infty}([0,\infty);X),\quad z_0\in\mathcal{D}(\cA).\]
Furthermore, the following statements are valid for $T>0$:}
\begin{enumerate}[(a)]
\item\label{item:xuniq}
If $x\in W^{1,\infty}([0,T];X)$ is such that both $x(t)\in\mathcal{D}(\Af)$ and \eqref{eq:funnel_problem} hold for almost every $t\in[0,T]$, then for
	\begin{equation}\label{eq:change2}
z(t):=\varphi(t) \Big(x(t)-Q \yref(t)-P(\Bf x_{0}+\psi(\varphi_0,e_0))p(t)\Big),
	\end{equation}
 it holds that $z\in W^{1,\infty}([0,T];X)$ and both $z(t)\in\mathcal{D}(\mathcal{A})$ and
	\begin{equation}\label{eq:funnel_equation}
	\begin{aligned}
	\dot{z}(t)&=\cA (z(t))+(\omega(t)+\alpha)z(t)+f(t),\\
	z(0)&=z_0 .
	\end{aligned}
	\end{equation}
	holds for almost every $t\in[0,T]$. If additionally,  $\dot{x}$ is continuous for almost every $t\in[0,T]$, then $\dot{z}$ is continuous for almost every $t\in[0,T]$.
\item\label{item:zuniq} Conversely, if $z\in W^{1,\infty}([0,T];X)$ is such that both $z(t)\in\mathcal{D}(\cA)$ and  \eqref{eq:funnel_equation} holds for almost every $t\in[0,T]$,
then for
	\begin{equation}\label{eq:change}
x(t)=\varphi(t)^{-1}z(t)+Q \yref(t)+P(\Bf x_{0}+\psi(\varphi_0,e_0))p(t),
	\end{equation}
we have that $x\in W^{1,\infty}([0,T];X)$ and both $x(t)\in\mathcal{D}(\Af)$ and \eqref{eq:funnel_problem} hold for almost every $t\in [0,T]$.
If additionally  $\dot{z}$ is continuous for almost every $t\in[0,T]$, then $\dot{x}$ is continuous for almost every $t\in[0,T]$.
\item\label{item:z3} If $z\in W^{1,\infty}([0,\infty);X)$, then $x$ as in \eqref{eq:change} fulfills $x\in W^{1,\infty}([0,\infty);X)$.
\end{enumerate}
\end{proposition}
}
\begin{proof}
The statements $\omega\in W^{1,\infty}([0,\infty))$, $f\in W^{1,\infty}([0,\infty);X)$ follow from the product rule for weak derivatives \cite[p.~124]{Alt16}. Since $P$ maps to $D(\Af)$, we have $z_0\in\mathcal{D}(\Af)$. Further, by using $\Bf P=\Cf Q=I$, $\Bf Q=\Cf P=0$ {and Assumption \ref{ass:contr}, we obtain that $\|\Cf z_{0}\|=\varphi(0)\|\Cf x_{0}-y_{\mathrm{ref}}(0)\|<1$ and }
\[\phi(\Cf z_0)=\dfrac{\varphi_0\, e_0}{1-\varphi_0^2\,\|e_0\|^2}=-\Bf z_0,
\]
whence $z_0\in\mathcal{D}(\cA)$.\\
To prove \eqref{item:xuniq}, assume that $x\in W^{1,\infty}([0,T];X)$ has a derivative which is continuous and in the domain of $\Af$ almost everywhere. First note that the twice weak differentiability of $p$ and $\varphi$ together with the fact that $P$ and $Q$ map to $D(\Af)$ implies that $z\in W^{1,\infty}([0,T];X)$ with ${z}(t)$ being in $D(\Af)$ for almost every $t\in[0,T]$. By further using that \eqref{eq:funnel_problem} holds for almost every $t\in[0,T]$, we obtain ---analogously to the above computations for $z_0$--- that
\[\phi(\Cf z(t))=\dfrac{\varphi(t) \, e(t)}{1-\varphi(t)^2\| e(t)\|^2}=-\Bf z(t),
\]
which implies that $z(t)\in\mathcal{D}(\cA)$ for almost every $t\in[0,T]$. Further, a~straightforward calculation shows that \eqref{eq:funnel_problem} implies that $z(t)$ fulfills
\eqref{eq:funnel_equation}.\\
Statement \eqref{item:zuniq} follows by an argumentation straightforward to that in the proof of \eqref{item:xuniq}. Statement \eqref{item:z3} is a simple consequence of
$\inf_{t\ge0}\varphi(t)>0$, $\varphi,p\in W^{2,\infty}([0,\infty))$, $\yref\in W^{2,\infty}([0,\infty),{U})$ and the product rule for weak derivatives.
\end{proof}

The previous proposition is indeed the key step to prove Theorems~\ref{thm:funnel1} \& \ref{thm:funnel2} on the feasibility of the funnel controller. By using the state transformation \eqref{eq:change2} with inversion \eqref{eq:change}, the analysis of feasibility of the funnel controller reduces to the proof of existence of a solution to the nonlinear evolution equation \eqref{eq:funnel_equation} in which the time-dependence is now extracted to the inhomogeneity. This is subject of the following result, which is a slight generalization of \cite[Thm.~IV.4.1]{Showalter}, where equations of type \eqref{eq:funnel_equation} with constant $\omega$ and m-monotone $\cA$ are considered. {For that we will use Kato's results \cite[Thms.~1-3 and the the subsequent remark]{Kato}.
Note as well, that in \cite{Kato} the notion of m-monotonic operators $A$ is used, which means that $-A$ is m-dissipative.}

\begin{lemma}\label{lem:gen_kato}
	Let $T>0$, $X$ be a Hilbert space and $A:\cD(A)\subset X\to X$ be m-dissipative in $X$ with $0\in\cD(A)$ and $A(0)=0$.
		Then for each $z_0\in\cD(A)$, real-valued $\omega\in W^{1,\infty}([0,T])$ and $f\in W^{1,\infty}([0,T];X)$ there exists a unique $z\in W^{1,\infty}([0,T];X)$ with
	\begin{enumerate}[(i)]
		\item $z(t)\in\cD(A)$ for almost every $t\in[0,T]$;
		\item for almost every $t\in[0,T]$ it holds that
		\begin{align}\label{eq:abstractCP}
		\dot{z}(t)&=A(z(t))+\omega(t)z(t)+f(t),\quad t\in[0,T],\\
		z(0)&=z_0,\nonumber
		\end{align}
		\item $\dot{z}$ and $A(z)$ are continuous except at a countable number of values in $[0,T]$.
	\end{enumerate}
\end{lemma}

\begin{proof}
	Define the operator $\cA(t)z:=A(z)+\omega(t)z+f(t)$ for $(t,z)\in[0,T]\times\cD(A)$ with $\cD(\cA(t))=\cD(A)$ for all $t\in[0,T]$.
	{The assertions immediately follow from  \cite[Thms.~1-3]{Kato} and the subsequent remark in \cite{Kato} provided that the following two conditions are satisfied for $\cA(t)$:
	\begin{enumerate}
	\item there exists $\lambda>0$ such that $-\lambda I+\cA(t)$ is m-dissipative for all $t\in[0,T]$, and
	\item there exists $C>0$ such that for all $\mathcal{D}(\cA)$ and  $s,t\in[0,T]$,
	\begin{equation}\label{eq:katoproof}
	\|\cA(t)z-\cA(s)z\|\leq C|t-s|(1+\|z\|).
	\end{equation}
	\end{enumerate}
	Let $\lambda>\|\omega\|_{L^\infty}$. Then the dissipativity in the first condition is trivial. To show the range condition, let $\mu>0$ and $u\in X$. Then
	\[\mu z- (-\lambda z+\cA(t)z)=u\]
	can be rewritten as
	\[(\mu+\lambda-\omega(t))z-A(z)=u+f(t).\]
	Since $\lambda-\omega(t)>0$ uniformly in $t$ and $A$ is m-dissipative, it follows that there is a unique $z\in\cD(A)$ such that $\mu z-(-\lambda z+\cA(t)z)=u$. \\
	The inequality \eqref{eq:katoproof} simply follows from
	\begin{align*}
		\|\cA(t)z-\cA(s)z\|&=\|(\omega(t)-\omega(s))z+f(t)-f(s)\|\\
		&\leq C\,(\|\dot{\omega}\|_{L^\infty}\|z\|+\|\dot{f}\|_{L^\infty})|t-s|
	\end{align*}
	with $C=\max\{\|\dot{\omega}\|_{L^\infty},\|\dot{f}\|_{L^\infty}\}$.
	}
\end{proof}

\begin{proof}[Proof of Theorem \ref{thm:funnel1}]
Let $T>0$, and consider the nonlinear operator $\cA$ as in \eqref{eq:closed_loop_op} and $\omega\in W^{1,\infty}([0,\infty))$, $f\in W^{1,\infty}([0,\infty);X)$ and $z_0\in\mathcal{D}(\cA)$ as in~\eqref{eq:data}. Then Lemma~\ref{lem:gen_kato} implies that the nonlinear evolution equation \eqref{eq:funnel_equation} has a~unique solution $z\in W^{1,\infty}([0,T];X)$ in the sense that for almost all $t\in[0,T]$ it holds that $z(t)\in\mathcal{D}(\cA)$, $\dot{z}$ is continuous at $t$, and \eqref{eq:funnel_equation} is satisfied. Then
Proposition~\ref{lem:transf}\ref{item:zuniq}) yields that $x\in W^{1,\infty}([0,T];X)$ as in \eqref{eq:change} has the desired properties.\\
It remains to show uniqueness: Assume that $x_i\in W^{1,\infty}([0,T];X)$ are solutions of the closed-loop system \eqref{eq:funnel_problem} for $i=1,2$. Then
\begin{equation}z_i(t)=\varphi(t) \Big(x_i(t)-Q \yref(t)-P(\Bf x_{0}+\psi(\varphi_0,e_0))p(t)\Big),
\label{eq:xizi}\end{equation}
fulfills $\dot{z}_i(t)=\cA (z_i(t))+(\omega(t)+\alpha)z(t)+f(t)$ with $z_i(0)=z_0$, and the uniqueness statement in Lemma~\ref{lem:gen_kato} gives $z_1=z_2$. Now resolving \eqref{eq:xizi} for $x_i$ and invoking $z_1=z_2$ gives $x_1=x_2$.\\
{To see that the set $M=\{x_{0}\in D(\Af)\:\colon\:\varphi(0)\|\Cf x_{0}-y_{\mathrm{ref}}(0)\|<1\}$ is dense in $X$, let $\tilde{x}\in D(\Af)$ be such that $\Cf \tilde{x}=\yref(0)$, which exists by Assumption \ref{ass:system_class}\eqref{ass:CBonto}. Since $\ker \Cf$ is dense in $X$ as a consequence of Assumption \ref{ass:system_class}\eqref{ass:1}\&\eqref{ass:3}, see also Remark \ref{rem:system_class}\eqref{rem:system_class1}, it follows that  $\tilde{x}+\ker\Cf\subset M$ is dense.}
\end{proof}
It remains to prove  Theorem~\ref{thm:funnel2} which states that the global solution and its derivative are bounded in case of negativity of the constant $\alpha$ in Assumption~\ref{ass:system_class}~\eqref{ass:passivity}. To this end we employ one~further auxiliary result, which is a variant of Gr\"onwall's inequality.

\begin{lemma}{\cite[Lem.~IV.4.1]{Showalter}}\label{lem:groenwall}
	Let $a,b\in L^1([0,T])$ be real-valued with $b\geq0$ almost everywhere and let the absolutely continuous function $v:[0,T]\to(0,\infty)$ satisfy
	$$(1-\rho)\dot{v}(t)\leq a(t)v(t)+b(t)v(t)^\rho\quad\mbox{ for almost every }t\in[0,T],$$
	where $0\leq\rho<1$. Then
	$$v(t)^{1-\rho}\leq v(0)^{1-\rho}\, e^{\int_0^ta(s)\ds{s}}+\int_0^te^{\int_s^t a(r)\ds{r}}b(s)\ds{s}\quad\forall\, t\in[0,T].$$
\end{lemma}
%
Now we are ready to formulate the proof of Theorem \ref{thm:funnel2}.
\begin{proof}[Proof of Theorem \ref{thm:funnel2}]
Let $\alpha<0$, let $\cA$ be the nonlinear operator in \eqref{eq:closed_loop_op} and $\omega\in W^{1,\infty}([0,\infty))$, $f\in W^{1,\infty}([0,\infty);X)$ and $z_0\in\mathcal{D}(\cA)$ be as in~\eqref{eq:data}. \\
{\em Step 1:}  We show that the solution
$z\in W^{1,\infty}_{\rm loc}([0,\infty);X)$ of \eqref{eq:funnel_equation} (which exists by Lemma~\ref{lem:gen_kato}) is bounded:\\
We estimate that for almost all $t\geq 0$,
	\begin{align*}
	\tfrac{1}{2}\ddt\|z(t)\|_X^2\!\!\!\!\!\!\!\!\!\!\!\!&\;\;\;\;\;\;\;=\re \langle z(t), \dot{z}(t)\rangle_X\\
=&\,\re \langle z(t), \cA z(t)+(\omega(t)+\alpha)z(t)+f(t)\rangle_X\\
\leq &\,\re \langle z(t), \cA z(t)\rangle_X+(\omega(t)+\alpha) \|z(t)\|_X^2+\|z(t)\|_X\|f(t)\|_X\\
\stackrel{\eqref{eq:closed_loop_op}}{=}&\,\re \langle z(t), \Af z(t)\rangle_X+\omega(t) \|z(t)\|_X^2+\|z(t)\|_X\|f(t)\|_X\\
\stackrel{\eqref{eq:diss}}{\leq}&\,\re \langle \Bf z(t), \Cf z(t)\rangle_{{U}}+\alpha\|z(t)\|_X^2+\omega(t)\|z(t)\|_X^2+\|z(t)\|_X\|f(t)\|_X\\
\stackrel{\eqref{eq:closed_loop_op}}{=}&\,-\frac{\|\Cf z(t)\|^2}{1-\|\Cf z(t)\|^2}+\alpha\|z(t)\|_X^2+\omega(t)\|z(t)\|_X^2+\|z(t)\|_X\|f(t)\|_X\\
	\leq&\,\alpha\|z(t)\|_X^2+\omega(t)\|z(t)\|_X^2+\|z(t)\|_X\|f(t)\|_X.
	\end{align*}
	Now applying Lemma \ref{lem:groenwall} with $\rho=1/2$, using that the definition of $\omega$ in \eqref{eq:data}, leads to $\omega=\ddt \log(\varphi)$ and, by setting $\varepsilon:=\inf_{t\geq0}\varphi(t)>0$, we obtain that, for all $t\ge0$,
	$$\|z(t)\|_X\leq\varepsilon^{-1}\|z_0\|_X\varphi(t)e^{\alpha t}+\varphi(t)e^{\alpha t}\int_0^t{\varphi(s)^{-1}\, e^{-\alpha s}\|f(s)\|_X\ds{s}}.$$
	By the definition of $f$ in \eqref{eq:data}, there exist $c_0,c_1>0$ such that, for almost all $t\geq 0$, we have that $\|f(t)\|_X\leq \varphi(t)(c_0+c_1\|\yref\|_{W^{1,\infty}})$. Thus,
	$$\|z(t)\|_X\leq\varepsilon^{-1}\|z_0\|_X\varphi(t)e^{\alpha t}-\alpha^{-1}\varphi(t)(c_0+c_1\|\yref\|_{W^{1,\infty}})(1-e^{\alpha t}),$$
	whence $z\in L^\infty([0,\infty);X)$.\\
{\em Step 2:}  We show that $\dot{z}\in L^\infty([0,\infty);X)$:\\
To this end, let $h>0$ and, by using the dissipativity of $\cA$, consider
	\begin{align*}
	\tfrac{1}{2}\ddt\|z(t+h)-z(t)\|_X^2\leq&\;\alpha\|z(t+h)-z(t)\|_X^2+\omega_0(t+h)\|z(t+h)-z(t)\|_X^2\\
	&+|\omega_0(t+h)-\omega_0(t)|\|z(t)\|_X\|z(t+h)-z(t)\|_X\\
	&+\|f(t+h)-f(t)\|_X\|z(t+h)-z(t)\|_X.
	\end{align*}
	Again applying the Gr\"onwall type inequality from Lemma \ref{lem:groenwall} with $\rho=1/2$, dividing by $h$ and letting $h\rightarrow0$ yields
	\begin{align*}
	&\|\dot{z}(t)\|_X\\\leq&\ \varepsilon^{-1}\|\dot{z}(0)\|_X\varphi(t)e^{\alpha t}+\varphi(t)e^{\alpha t}\int_0^t{e^{-\alpha s}\varphi(s)^{-1}(\|z\|_{L^\infty} |\dot{\omega_0}(s)|+\|\dot{f}(s)\|_X)\ds{s}}\\
	\leq&\ \varepsilon^{-1}\|\cA(z_0)+\omega_0(0)z_0+f(0)\|_X\|\varphi\|_{L^\infty}\\
	&+\|\varphi\|_{L^\infty}\left\|\varphi^{-1}\right\|_{L^\infty}(\|z\|_{L^\infty}\|\dot{\omega_0}\|_{L^\infty}+d_0+d_1\|\yref\|_{W^{2,\infty}})
	\end{align*}
	for some $d_0,d_1>0$. Hence, $\dot{z}(t)\in L^\infty([0,\infty);X)$.\\
{\em Step 3:} We conclude that the solution $x$ in \eqref{eq:funnel_problem} (which exists by Theorem \ref{thm:funnel1}) fulfills $x\in W^{1,\infty}([0,\infty);X)$:\\
We know from the first two steps that $z\in W^{1,\infty}([0,\infty);X)$. Then Proposition~\ref{lem:transf}\eqref{item:z3} leads to $x\in W^{1,\infty}([0,\infty);X)$.\\
{\em Step 4:} We finally show that $u=\Bf x$ fulfills $u\in L^{\infty}([0,\infty);{U})$:\\
We know from the third step, we know that $x\in W^{1,\infty}([0,\infty);X)$. Since we have $x(t)\in\mathcal{D}(\Af)$ with $\dot{x}(t)=\Af x(t)$ for almost all $t\ge0$, we can conclude that $\Af x\in L^{\infty}([0,\infty);X)$, and thus $x\in L^{\infty}([0,\infty);D(\Af))$. Then $\Bf\in\cL(D(\Af),{U})$ gives $u=\Bf x\in L^{\infty}([0,\infty);{U})$.
\end{proof}


\section{Simulations}
\label{sec:sim}
Here we show some examples which correspond to the classes mentioned in Section \ref{sec:main}. The implementation of all simulations has been done with Python.

\subsection{Lossy transmission line}
We consider the dissipative version of the Telegrapher's Equation with constant coefficients given by
\begin{align*}
\tfrac{\partial V}{\partial \zeta}(\zeta,t)&=-L\tfrac{\partial I}{\partial t}(\zeta,t)-RI(\zeta,t),\\
	\tfrac{\partial I}{\partial \zeta}(\zeta,t)&=-C\tfrac{\partial V}{\partial t}(\zeta,t)-GV(\zeta,t),\\
u(t)&=\begin{pmatrix}
V(a,t)\\
V(b,t)
\end{pmatrix},\qquad
y(t)=\begin{pmatrix}
I(a,t)\\
-I(b,t)
\end{pmatrix}.
\end{align*}
$R$ is the resistance, $C$ the capacitance, $L$ the inductance and $G$ the conductance ---all of them per unit length.

The system can be written in port-Hamiltonian form as
\begin{align*}
\tfrac{\partial x}{\partial t}(\zeta,t)&=P_1\tfrac{\partial}{\partial \zeta}(\cH(\zeta)x(\zeta,t))+P_0\cH(\zeta)x(\zeta,t),\\
u(t)&=W_BR_0\begin{pmatrix}
(\cH x)(b,t)\\
(\cH x)(a,t)
\end{pmatrix},\qquad
y(t)=W_CR_0\begin{pmatrix}
(\cH x)(b,t)\\
(\cH x)(a,t)
\end{pmatrix},
\end{align*}
where
\begin{align*}
&P_1:=\begin{bmatrix}
0 & -1\\
-1 & 0
\end{bmatrix},\;\;P_0:=\begin{bmatrix}
-R & 0\\
0 & -G
\end{bmatrix},\;\;\cH(\zeta):=\begin{bmatrix}
L^{-1} & 0\\
0 & C^{-1}
\end{bmatrix},\\
\\
&W_B:=\dfrac{1}{\sqrt{2}}\begin{bmatrix}
1 & 0 & 0 & 1\\
-1 & 0 & 0 & 1
\end{bmatrix},\;\; W_C:=\dfrac{1}{\sqrt{2}}\begin{bmatrix}
0 & 1 &  1 & 0\\
0 & 1 & -1 & 0
\end{bmatrix},\\
\\
&x(\zeta,t):=\begin{pmatrix}
LI(\zeta,t)\\
CV(\zeta,t)
\end{pmatrix}.
\end{align*}
We have chosen the reference signals and funnel boundary of the following form
\begin{align*}
\yref(t)&=\begin{pmatrix}
A_1\sin(\omega_1t)\sin(\omega_2t)\\
A_2\sin(\omega_3t)
\end{pmatrix},\\
\varphi(t)&=\varphi_0\varepsilon^{-2}\tanh(\omega t+\varepsilon).
\end{align*}
In this case the system is impedance passive and $P_0+P_0^*\leq-2\min\{R,G\} I_2$ and Theorem \ref{thm:funnel_phs} implies that $u,y\in L^\infty([0,\infty);\R^2)$. The simulated system is shown in Fig.~\ref{fig:1}.

\begin{figure}
	\centering
	\includegraphics[width=300pt]{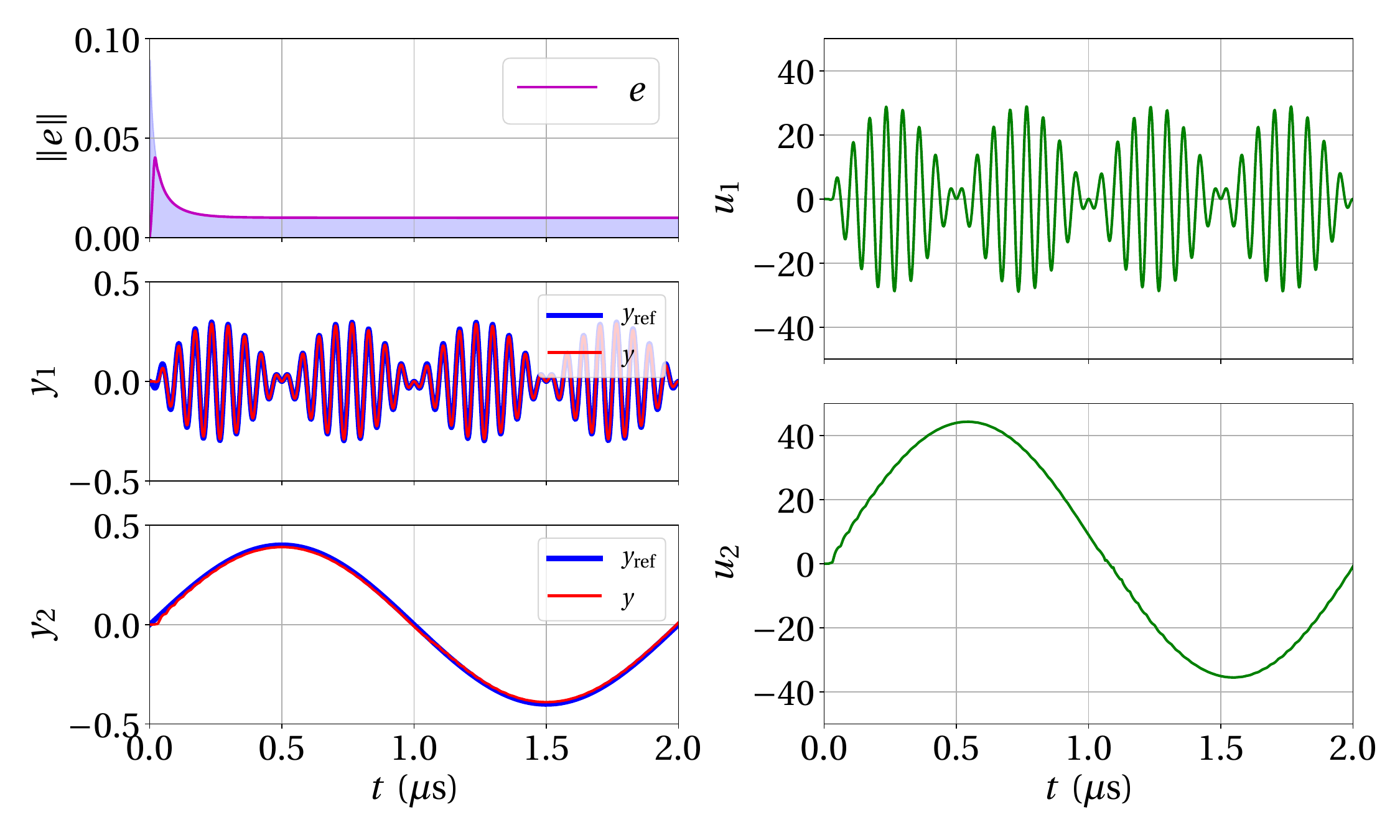}
	\caption{Left: Norm of the error within the funnel boundary followed by the two reference signals and the respective outputs. Right: Inputs obtained from the feedback law.}
	\label{fig:1}
\end{figure}

The parameter values are $\zeta\in(a,b)$ with $a=0\ {\rm m},b=1\ {\rm m}$,
\[
\begin{array}{ccccccc}
R&=&463.59 \Omega {\rm m^{-1}},& &L&=&0.5062\ {\rm mH\ m^{-1}},\\
G&=&29.111\ \mu{\rm S\ m^{-1}},& &C&=&51.57\ {\rm nF\ m^{-1}},
\end{array}
\]
{which correspond to representative parameter data for $24$ gauge telephone polyethylene insulated cable (PIC) at $294\ \mathrm{K}$ with a frequency $f=1\ {\rm MHz}$}. Further, set $c_0=(LC)^{-1/2}$, $\omega=2\pi f$, $\varphi_0=1\ {\rm A^{-1}}$, $\varepsilon=0.1$ and amplitudes $A_1=-0.3\ {\rm A}, A_2=0.4\ {\rm A}$. The other angular frequencies are $\omega_1=\omega$, $\omega_2=16\omega$ and $\omega_3=\omega/2$. For the time interval we have defined $T_0=f^{-1}$ and $t\in[0,T]$, where $T=2T_0$. We have used semi-explicit finite differences with a tolerance of $10^{-3}$. The mesh in $\zeta$ has $M=1000$ points and the mesh in $t$ has
\[N=\left\lfloor\frac{b-a}{2c_0T}M\right\rfloor\]
points.
We further assume that the initial state is zero, i.e.,
$x_0=0$ and we apply the controller \eqref{eq:funnel_controllerk0} from Remark~\ref{rem:k0} with $k_0=1\ \Omega$.

\subsection{Heat equation}

Here we consider the following boundary controlled 2D heat equation on the unit disc given by
\begin{align*}
\tfrac{\partial x}{\partial t}(t,r,\theta)&=\alpha\left(\tfrac{\partial^{2}x}{\partial r^{2}}(t,r,\theta)+r^{-1}\tfrac{\partial x}{\partial r}(t,r,\theta)+r^{-2}\tfrac{\partial^{2}x}{\partial \theta^{2}}(t,r,\theta)\right),\\
u(t)&=\alpha\left(\tfrac{\partial x}{\partial r}(t,r,\theta)\right)|_{r=1},\qquad
y(t)=\int_0^{2\pi}{x(t,1,\theta)\mathrm{d}\theta},
\end{align*}
where $\alpha>0$ is the thermal diffusivity.

In this case, applying Theorem \ref{thm:funnel_heat}, we choose a funnel boundary of the form $\varphi(t)=\varphi_0\varepsilon^{-2}\tanh(\omega t+\varepsilon)$. The reference signal is given by $\yref(t)=A\sin(\omega t)$ and the simulated system is shown in Fig.~\ref{fig:3}. In Fig.~\ref{fig:4} we show the evolution of the plate at four different times.

\begin{figure}
	\centering
	\includegraphics[width=200pt]{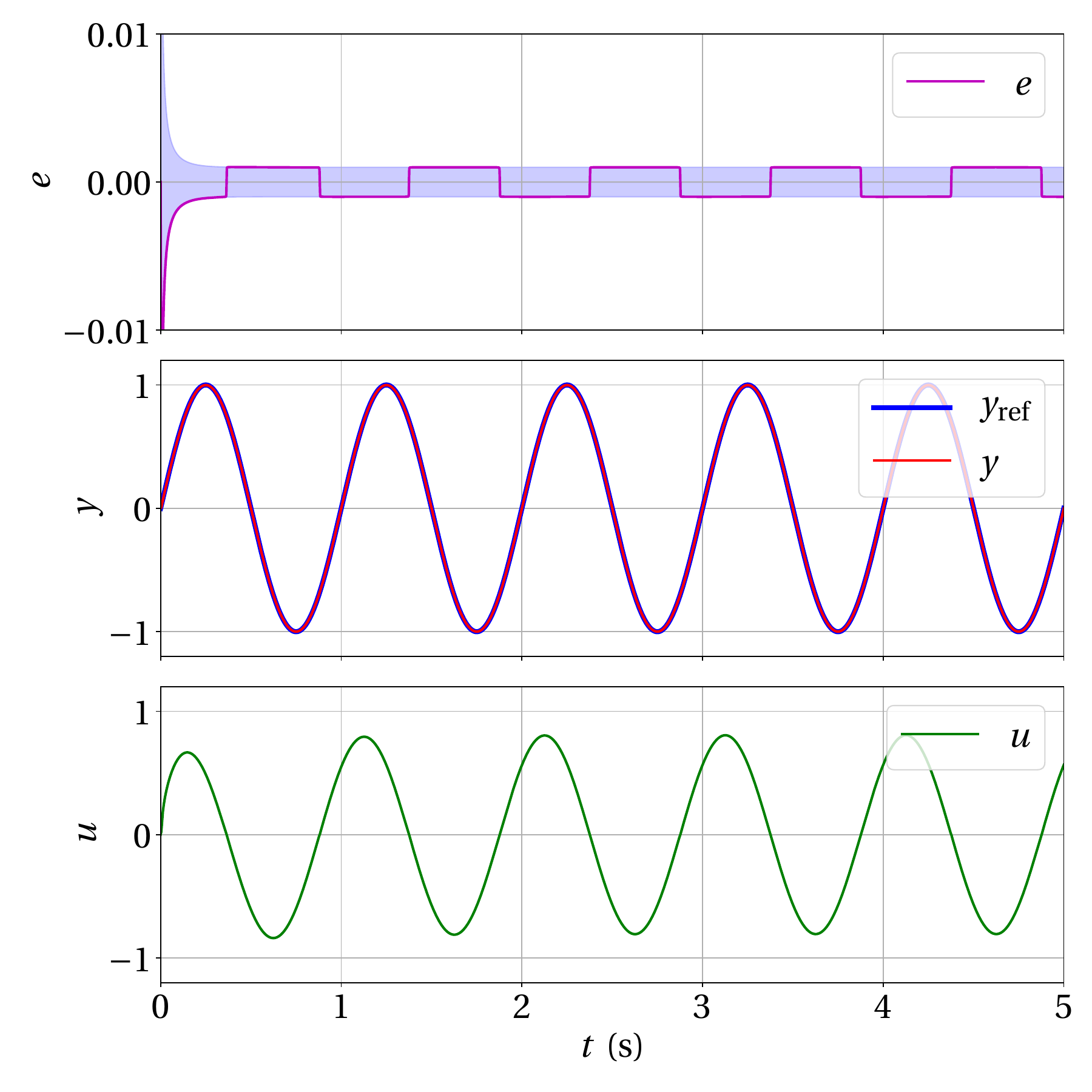}
	\caption{Performance funnel with the error, reference signal with the output of the closed-loop system and input of the closed-loop.}
	\label{fig:3}
\end{figure}

\begin{figure}
	\centering
	\includegraphics[width=200pt]{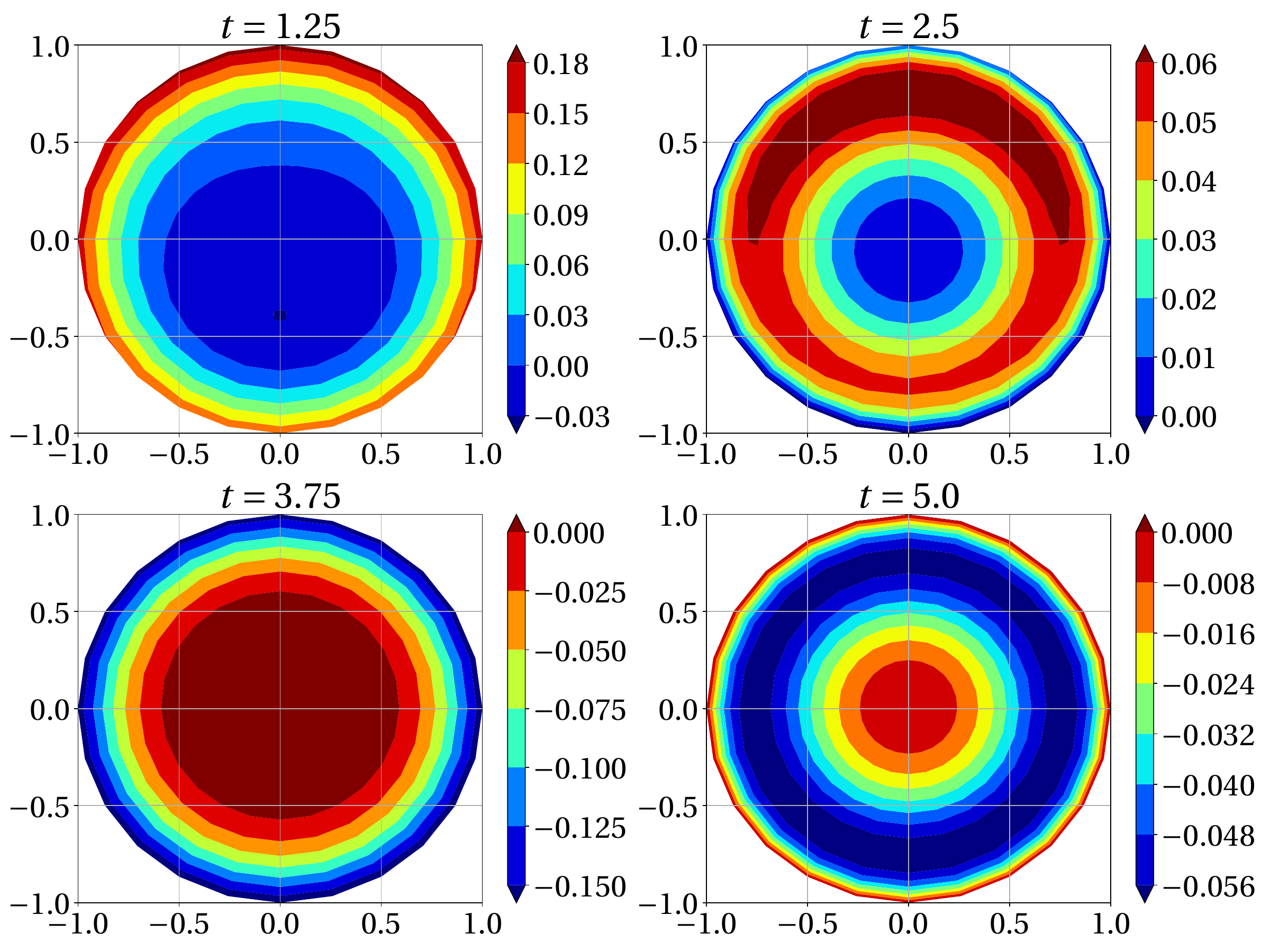}
	\caption{From left to right, top to bottom, the temperature of the plate for different increasing times.}
	\label{fig:4}
\end{figure}

The parameter values are $\alpha=1\ {\rm m^2s^{-1}}$, $r\in(r_0,r_1)$, with $r_0=0\ {\rm m}$ and $r_1=1\ {\rm m}$, and $\theta\in(0,2\pi)$. The amplitude values are $A=1\ {\rm J}$, $\varphi_0=0.1\ {\rm J}$ and $\varepsilon=10^{-1}$. We have set $T_0=1\ {\rm s}$, $\omega=2\pi T_0^{-1}$ and $T=5T_0$. We have used explicit finite differences with a partition in $r$ and $\theta$ of $N=25$ points for each variable and in $t\in[0,T]$ of
$$M=\left\lfloor10T\left(\frac{N^2}{(r_1-r_0)^2}+\frac{N}{r_1-r_0}+\frac{N^2}{4\pi^2}\right)\right\rfloor$$
points. The initial state of the system is
$$x(0,r,\theta)=x_0(r_1-r)^2\sin(\theta),$$
where
$x_0=0.5\ {\rm J m^{-2}}$. We apply the controller \eqref{eq:funnel_controllerk0} from Remark~\ref{rem:k0} with $k_0=1\ {\rm s^{-2}}$.

\end{document}